\documentclass[12pt,reqno]{amsart}
\usepackage{a4wide}

\numberwithin{equation}{section}

\allowdisplaybreaks

\usepackage{color}

\newtheorem{theorem}{Theorem}[section]
\newtheorem{proposition}[theorem]{Proposition}
\newtheorem{corollary}[theorem]{Corollary}
\newtheorem{lemma}[theorem]{Lemma}

\theoremstyle{definition}

\theoremstyle{remark}
\newtheorem{remark}[theorem]{Remark}

\renewcommand{\epsilon}{\varepsilon }

\newcommand{\R}{\mathbb{R}}

\begin{document}
\title[Existence, local uniqueness of bubbling solutions ]%
{ existence and  local uniqueness of bubbling solutions for poly-harmonic equations with critical growth  }

\author{  Yuxia Guo, Shuangjie Peng and Shusen Yan}

\address{  Department of Mathematical Science, Tsinghua University}
\email{yguo@math.tsinghua.edu.cn}

\address{School  of Mathematics and Statistics,
  Central  China Normal University,
  Wuhan, P. R. China }
\email{sjpeng@mail.ccnu.edu.cn}

\address{
Department  of Mathematics,  The
University of New
England,  Armidale, NSW 2351,   Australia}
\email{syan@turing.une.edu.au}

\subjclass{Primary 35B40, 35B45; Secondary 35J40}

\keywords {Multi-bubbling solutions, Poly-harmonic equations,  Critical exponents, Prescribed scalar curvature,  Local uniqueness, Symmetry}

\date{}


\begin{abstract}
We consider the following poly-harmonic equations with critical exponents
:

\begin{equation}\label{P}
 (-\Delta)^m u
=K(y)u^{\frac{N+2m}{N-2m}},\;\;\; u>0\;\;\;\hbox{
in } \mathbb{R}^N,
\end{equation}
where $N> 2m+2,m\in\mathbb{N}_{+}, K(y)$ is positive and
periodic in its  first $k$ variables $(y_1,\cdots, y_k)$, $1\leq k<\frac{N-2m}{2}$.
Under some conditions on $K(y)$ near its critical point, we  prove not only  that
problem~\eqref{P} admits solutions with infinitely many bubbles, but also that
    the bubbling solutions obtained  in our existence result are locally unique. This local
    uniqueness result implies that some bubbling solutions preserve the symmetry of the scalar curvature $K(y).$

\end{abstract}

\maketitle


\section{Introduction}

 We consider the following poly-harmonic equations with critical exponent:
$$
(-\Delta)^m u
=K(y)u^{\frac{N+2m}{N-2m}},\;\;\; u>0\;\;\;\hbox{
in } \mathbb{R}^N,\eqno{(P)}
$$
where $N > 2m+2,m\in\mathbb{N}_{+}$,  and $ K(y)$ is a bounded positive smooth function.

 When $m=1$,
 problem~$(P)$  is the prescribed scalar curvature problem in $\mathbb R^N$.
 It is also well-known that  a solution of the following problem:
\begin{equation}\label{sn}
\begin{cases}
-\Delta u
=K(y)u^{\frac{N+2}{N-2}},\;\;\; u>0& \hbox{
in } \mathbb{R}^N, \\
u\in \mathcal{D}^{1,2}(\mathbb{R}^{N}),
\end{cases}
\end{equation}
solves the
 prescribed scalar curvature problem  on $\mathbb{S}^N$.

\vskip8pt

 From the Pohozaev identity, it is easy to  see that   problem \eqref{sn} does
not always admit a solution.  We are interested in the sufficient
conditions on $K(y)$, under which \eqref{sn} admits a solution. In the last three decades,  there have been
considerable interests in the existence and multiplicity of solutions for problem
\eqref{sn} under some  suitable assumptions on the function
$K(y)$. See for example, \cite{AAP}, \cite{CNY},  \cite{L1},
\cite{Y} and the references therein. When $K(y)$ is positive and periodic, by gluing approximation solutions into genuine solutions which concentrate  at some  isolated  maximum points of the function $K(y)$, Li \cite{L1},
\cite{L2}, \cite{L3} proved that \eqref{sn} has infinitely many multi-bubbling solutions for $N\geq 3$ (see  \cite{Y} for the more general results).  When $K(y)$ is a positive radial function with a strict local maximum at $|x|=r_0>0$ and satisfies
$$K(r)=K(r_0)-c_0|r-r_0|^\beta+O(|r-r_0|)^{\beta+\theta},\quad r\in (r_0-\delta, r_0+\delta),$$
for some constants $c_0>0$, $\beta\in [2, N-2)$,  and small constants
$\theta>0$, $\delta>0$,  Wei and Yan \cite{WY} constructed solutions for  \eqref{sn} with large number of bubbles  concentrating near the sphere $|x|=r_0$ for $N\geq 5.$
Very recently, Li, Wei and Xu \cite{LWX} proved the existence
 of solutions with infinitely many bubbles for problem $(P)$, where the centers of the bubbles can be placed on
  all the $k$-dimensional lattice points with $k<\frac{N-2}2$. Moreover, they showed that the dimension restriction is optimal.
 For other  related problems with critical exponents, we refer to  \cite{AAP},
\cite{BC}, \cite{be}, \cite{CNY}, \cite{CL1}, 
\cite{L1}, \cite{L4},
\cite{yylinwm},
 \cite{LL}, 
 \cite{SZ}
and references therein.

\vskip8pt

 In recent years, the poly-harmonic operators have found considerable interest. For instance, when $m=2$,
problem (P) is related to the Paneitz operator, which was
introduced by Paneitz \cite{paneitz} for smooth $4$ dimensional
Riemannian manifolds  and was generalized by Branson \cite{branson} to
smooth $N$ dimensional Riemannian manifolds. We refer the reader
to the papers \cite{aschang}, \cite{bartschw}, \cite{bartsch},
\cite{efj}, \cite{ggs}, \cite{grunau}, \cite{gs1}, \cite{gs2},
\cite{ps1}, \cite{ps2}, 
  and the references therein, for various existence results on
the poly-harmonic operators and related  problems. One can see from these papers that the poly-harmonic operator presents new and
challenging features compared with the Laplace operator. To the best of our knowledge, not much is obtained
for the existence and the properties of bubbling solutions
for elliptic problems involving poly-harmonic operators and critical exponents.

The aim of this paper is two-fold.  Firstly,
  we will  construct solutions with infinitely many bubbles for  problem $(P)$ under some more reasonable
  conditions than those in \cite{LWX}. Secondly,
   we will study the properties  of the  bubbling solutions for  problem $(P)$, especially,
the periodicity property of these bubbling solutions. The problem for the symmetry of the bubbling
 solutions is  independently interesting
and is harder to study than the existence.  Obviously it
can not be solved by the methods of moving plane. Instead, we will attack it by studying the
local uniqueness of a sequence of bubbling solutions via various Pohozaev identities. Note that for general integer $m>0$, it is impossible
to estimate each term appearing in the Pohozaev identities. Thus, a better understanding
of the Pohozaev identities is essential in the proof of our local uniqueness result, which, we believe, will be very useful in the study of  other related problems.

\vskip8pt

 We assume that $K(x)$ satisfies the following conditions:

$(A_1)$  $0\le \inf_{\mathbb{R}^N}K(x)< \sup_{\mathbb{R}^N}K(x)<\infty;$

$(A_2)$  $K\in C^1({\mathbb{R}^N})$ is $1$-periodic in its first $k$ variables;

 $(A_3)$  $0$ is a critical point of $K$ and there exists some real numbers $\beta\in(N-2m, N)$ such that
 for all $|x|$ small, it holds

 $$
 K(x)=K(0)+\sum_{i=1}^N a_i|x_i|^\beta+R(x),
 $$
where $K(0)>0$, $a_i\not=0, \sum_{i=1}^N a_i<0, R(x)$ is $C^{[\beta], 1}({\mathbb{R}^N})$ near $0$ and $\sum_{s=0}^{[\beta]}|\nabla^s R(x)||x|^{-\beta+s}=O(|x|^\theta)$ for some $\theta>0$ as $x$ tends to $0,$ where
$C^{[\beta], 1}$ means that up to $[\beta]$ derivatives are Lipschitz functions, $[\beta]$ denotes the integer part of $\beta$, $\nabla^s$ denotes all the partial derivatives of order $s.$

\vskip8pt

  To state   the main  results of this paper,  we need to introduce  some notations first.  For any  integer $k\in [1, N],$ we define $k$-dimensional lattice by:
$$
Q_k:=\{\hbox{all the integer points in}\; \mathbb R^k\times\{0\} \subset \mathbb{R}^N\}, \quad\hbox{where } 0\in \mathbb{R}^{N-k}.
$$

 In this paper, we always assume that $k<\frac{N-2m}2$.
  Take any sequence of integers $\tilde P_j\in Q_k$, satisfying $\tilde P_i\ne \tilde P_j$ for $i\ne j$.
 It is easy to check
\[
  \sum_{i\ne j}\frac1{|\tilde P_i- \tilde P_j|^{\tau} }<+\infty,\quad \forall\; j,
\]
where $\tau= \frac{N-2m}2-\vartheta$, $\vartheta>0$ is a fixed small constant. With this choice of $\vartheta$,
we have  $\tau>k$.

 The condition we impose on the choice of the points $\tilde P_j\in Q_k$  is the following:
\begin{equation}\label{1-6-1}
 \max_j \sum_{i\ne j}\frac1{|\tilde P_i- \tilde P_j|^{\tau} } \le C\min_j \sum_{i\ne j}\frac1{|\tilde P_i- \tilde P_j|^{\tau} }<+\infty.
\end{equation}
 Let $L>0$ be a large integer.  Denote $P_{ j}= \tilde P_j L$. We are going to construct a bubbling solution,
 concentrating at $P_j$, $j=1, 2,\cdots$.  For this
 purpose,
we take $x_{j, L}$, which is close to $P_{ j}$,  $\mu_{j, L}>0$ large, and define
$$
U_{x_{j,L}, \mu_{j,L}}(y)=\tilde C_m \frac{\mu_{j,L}^{\frac{N-2m}{2}}}{(1+\mu_{j,L}^2|y-x_{j,L}|^2)^{\frac{N-2m}{2}}}$$
with $\tilde C_m=\bigl(\displaystyle\amalg_{h=-m}^{m-1}(N+2h)\bigr)^{\frac{N-2m}{4m}}$. Note that $U_{0,1}$ is the
unique solution (up to a translation and  a scaling) to the problem
(see \cite{wx}
):
$$
(-\Delta)^m u=u^{\frac{N+2m}{N-2m}}, u>0 \hbox{ in }\mathbb{R}^N.
$$
   Similar to \cite{LWX}, we define the following norms:
\begin{equation}\label{100-16-1}
\|\phi\|_{\ast}=\sup\limits_{y\in\mathbb{R}^{N}}\Big(\sigma(y)\sum\limits_{j=1}^{\infty}\frac{\mu_{j,L}^{\frac{N-2m}2}}{(1+\mu_{j,L}
|y-x_{j,L}|)^{\frac{N-2m}{2}+\tau}}\Big)^{-1}|\phi(y)|,
\end{equation}
\begin{equation}\label{101-16-1}
\|f\|_{\ast\ast}=\sup\limits_{y\in\mathbb{R}^{N}}\Big(\sigma(y)\sum\limits_{j=1}^{\infty}\frac{
\mu_{j,L}^{\frac{N+2m}2}}{(1+\mu_{j,L}|y-x_{j,L}|)
^{\frac{N+2m}{2}+\tau}}\Big)^{-1}|f(y)|,
\end{equation}
where $\sigma(y)=\hbox{min}\{1, \hbox{min}_{i=1}^{\infty}(\frac{1+\mu_{i, L}|y-x_{i, L}|}{\mu_{i, L}})^{\tau-1}\}$,
and $\tau>k$ is the same constant as in \eqref{1-6-1}.
In the following of this paper, we will also use the same notation $\|\cdot\|_*$ and $\|\cdot\|_{**}$
if the sum in \eqref{100-16-1} or \eqref{101-16-1} is from 1  to $n$.

\vskip8pt

Our first result is the following.

 \vskip8pt
 \begin{theorem}\label{th11}
  Suppose that  $K$ satisfies the conditions $(A_1), (A_2)$ and $(A_3)$. Assume that
  $N>2m+2,$  $1\le k<\frac{N-2m}{2}$,  and the sequence $P_{ j}= \tilde P_j L$ satisfies \eqref{1-6-1}.
  Then,  problem $(P)$ has a solution $u_L$, satisfying
 \begin{equation}\label{3-6-1}
   ||u_L-\sum_{i=1}^{\infty }U_{x_{i, L}, \mu_{i, L}}||_{*}=o_L(1),
   \end{equation}
 for some $x_{i,L} $ and $\mu_{i, L}$,   with
\begin{equation}\label{2-6-1}
 x_{i, L}=P_{ i}+o_L(1)
\end{equation}
 and
 \begin{equation}\label{4-6-1}
\mu_{i, L}  L^{-\frac{N-2m}{\beta-N+2m}}=O(1),
\end{equation}
where $o_L(1)\to 0$ as $L\to +\infty$.

\end{theorem}

\vskip8pt

Note that
\eqref{4-6-1} implies
\begin{equation}\label{2-1-9}
\mu_0\le \frac{\mu_{i, L}}{\mu_{j, L}} \le \mu'_0,\quad \forall \; i, \;j,
\end{equation}
 for some constants $\mu'_0> \mu_0>0$ which are independent
of $L$.

To discuss the symmetry properties of the solutions obtained in Theorem~\ref{th11}, we proceed with
 the following local uniqueness result for the bubbling solutions of $(P)$.

\vskip8pt
\begin{theorem}
\label{main} Under the same assumptions as in Theorem~\ref{th11}, if $u^{(1)}_{L}
$  and
 $u^{(2)}_{L}$ are two sequence of solutions of problem $ (P)$, which
 satisfy \eqref{3-6-1}, \eqref{2-6-1}  and \eqref{2-1-9},
 then
$u^{(1)}_{ L}= u^{(2)}_{ L}$ provided $L>0$ is large enough.
\end{theorem}

\vskip8pt

   Local uniqueness is an important topic in the study of elliptic partial equations. Theorem~\ref{main}  can be used to
study the properties of the bubbling solutions.
A direct consequence of this result is the following periodicity property of the solutions.

\vskip8pt
\begin{theorem}
\label{th1-7-1} Under the same assumption as in Theorem~\ref{th11}, if $\{\tilde P_j:\; j=1, 2,\cdots\}= Q_k$, and $u_{L}
$ is a solution of $(P)$, which
 satisfies \eqref{3-6-1}, \eqref{2-6-1}  and \eqref{2-1-9},
 then
$u_{ L}$ is $L-$periodic in $y_j$, $j=1, 2, \cdots, k$, provided $L>0$ is large enough.
\end{theorem}

\vskip8pt

  The construction of bubbling solutions for $(P)$ is somewhat standard. So in this paper, we will be
a bit sketchy in the proof of Theorem~\ref{th11}. Our main contribution to the existence result
is to find a more suitable condition \eqref{1-6-1} for $\tilde P_{j}$ in order to construct a bubbling solution blowing at the
given set $\{P_{j}= L\tilde P_{j} :\; j=1, 2, \cdots\}$. Condition \eqref{1-6-1} shows that it is not
necessary to take all the lattice point $Q_k$ to form the set $\{\tilde P_{j} :\; j=1, 2, \cdots\}$.
Of course, if we take all the points in $Q_k$, then \eqref{1-6-1} holds.

In order to prove the local uniqueness result, we will use various kinds of local Pohozaev identities.
 Note that in the case of  $m=1$ (studied in \cite{DLY}),  each integral appearing in the  local Pohozaev identities
can be calculated or estimated. However,  we can not follow the same procedure as in \cite{DLY} for general
integer $m>0$, because the number of the integrals appearing in the  local Pohozaev identities
approaches to infinity as $m\to +\infty$. To make the things even worse, it seems impossible to give
a precise local Pohozaev identities for general $m$. Thus,  a better understanding of all those
 local Pohozaev identities plays an essential role in the proof of Theorem~\ref{main}.  See the discussions
 in Section~3.

  Note that in \cite{WY,DLY}, the  center of the bubbles lies in a one dimensional space. When
 the center of the bubbles lies in a $k$ dimensional space, technical difficulties occur in
 the study of  both the existence and the local uniqueness of the bubbling solutions. For the
 existence, these difficulties were overcome in \cite{LWX} by introducing the weight function $\sigma$ in
 the norms defined in \eqref{100-16-1} and \eqref{101-16-1}. For the local uniqueness, it seems that the
  key lemma
 in \cite{DLY} does not hold anymore if $k$ is large. So new estimates need to be developed to deal with
 this case.

 Let point out that we can replace $(A_2)$ by the following condition: $K\in C^1({\mathbb{R}^N})$ is $l_i$-periodic in $y_i$, $i=1, \cdots, k$, for some $l_i>0$.  Under this new condition, we just need to define $Q_k$ as follows:

 $$
Q_k:=\{(j_1l_1, \cdots, j_k l_k, 0)\in \mathbb R^k\times\{0\} \subset \mathbb{R}^N,\;\;
\text{ for all integers $j_i$, $i=1, \cdots, k$}\}.
$$

\vskip8pt

 The paper is organized as follows.    In section 2, we study the existence of  bubbling solutions.
 Section 3 is devoted  to the discussion of  the local uniqueness and periodicity of a sequence of bubbling solutions.  In  Appendix A, some basic estimates are proved, while in Appendix B, we compute the formula for the asymptotic energy expansion. Some estimates of the error term are given in Appendix C. In Appendix D, we give some basic lemmas
 in algebra, which will be  used in the proof of our existence and local uniqueness results.

\section{Existence of solutions with infinitely many bubbles}

 In this section, we will prove Theorem~\ref{th11}. To this end,
  we firstly construct a bubbling solutions blowing-up at finite points $\{P_{1}, \cdots, P_{n}\}$.
Throughout this paper, we define $m^*=\frac{2N}{N-2m}$.

\vskip8pt

  \begin{theorem}\label{th2-7-1}

  Suppose that  $K$ satisfies the conditions $(A_1), (A_2)$ and $(A_3)$. Assume that
  $N>2m+2, $ $1\le k<\frac{N-2m}{2} $,  and the sequence $\{ \tilde P_j:\; j=1, \cdots, n\}$ satisfies \eqref{1-6-1},
  where the constant $C$ is independent of $n$. Then,  problem  $(P)$ admits  a solution $u_{n,L}$, satisfying
\begin{equation}\label{nn3-6-1}
   ||u_{n, L}-\sum_{i=1}^{n }U_{x_{n,i,L}, \mu_{n,i,L}}||_{*}=o_L(1),
   \end{equation}
 for some $x_{n,i, L} $ and $\mu_{n,i,L}$,   with
$ x_{n,i,L}=P_{i}+o_L(1)$
 and
$\mu_{n,i,L}  L^{-\frac{N-2m}{\beta-N+2m}}\le C,
$
where $C>0$ is independent of $n$, and $o_L(1)\to 0$ uniformly in $n$ as $L\to +\infty$.

\end{theorem}

\vskip8pt

\subsection{Linearization and finite dimensional reduction}

{For $i=1,\cdots,n$, let $x_i\in B_{1}(P_i),\,\,\mu_i>0$. Set $\mathbf {x}=(x_1,\cdots,x_n)$ and $\mathbf {\mu}=(\mu_1,\cdots,\mu_n)$ satisfying $0<\mu_0<\mu_i/\mu_j<\mu'_0<+\infty$.  For simplicity,  we denote
  {$U_{x_{i}, \mu_{i}}(x)$} by $U_i(x), $  $i=1, 2, \cdots,n.$  Let
\[
Z_{i,j}=\xi(y-x_{i})\frac{\partial U_{i}}{\partial
x_{i,j}},\;\;
Z_{i,N+1}=\xi(y-x_{i})\frac{\partial U_{i}}{\partial \mu_{i}},\quad j=1,\cdots,N,\ \ i=1, \cdots,n,
\]
where $\xi\in C_0^\infty(B_2(0))$ satisfying
$\xi(y)=\xi(|y|)$,   $\xi=1$ in $B_1(0)$, and $\xi=0$ in $\mathbb R^N\setminus B_2(0)$.
The purpose of  using  the cut-off function $\xi$ above is just to  make the calculations simpler.

We define the function spaces $\mathbf{X}$ and
    $\mathbf{Y}$ as follows:   $\phi\in \mathbf{X}$ if  $\|\phi\|_*<+\infty$, while $  f\in \mathbf{Y}$ if
    $\|f\|_{**}<+\infty$.
Set
\begin{equation}\label{100-15-1}
\mathbf{H}_n:=\Big\{\phi:\;\; \phi\in\mathbf {X},\;\int_{\mathbb{R}^N}\phi U_i^{\frac{4m}{N-2m}}Z_{i, j}=0, \;\; i=1, \cdots,n,\; j=1, \cdots, N
+1\Big\}.
\end{equation}

 Let
 {$W_{n}(y)=\sum_{i=1}^{n }U_{i}(y),$}   $\phi\in \mathbf{H}_n$.
 We want to find a  solution of the form  $W_{n}(x)+\phi(x)$
 for $(P)$ with $\|\phi\|_*$ small. To achieve this goal,  we first prove that for fixed $({\bf x, \mu}),$
 there exists a smooth function $\phi\in \mathbf{H}_n$, such that
\begin{equation}\label{j3}
(-\Delta )^m  (W_n(x)+\phi(x))- K(y) (W_n(x)+\phi(x))_+^{m^*-1}
=\sum_{i=1}^n\sum_{j=1}^{N+1}c_{ij} U_i^{m^*-2} Z_{i,j}
\end{equation}
for some constants $c_{ij}$.  Then, we show the existence of $({\bf x, \mu}),$ such that
\begin{equation}\label{nnj3}
\int_{\mathbb R^N} (-\Delta )^m (W_n(x)+\phi(x)) Z_{i,j}-\int_{\mathbb R^N} K(y) (W_n(x)+\phi(x))_+^{m^*-1}
 Z_{i,j}
=0.
\end{equation}
With this $({\bf x, \mu}),$   it is easy to prove that all  $c_{ij}$ must be zero.

\vskip8pt

 {\bf Part I: The Reduction.}  In this part, for fixed {$({\bf x, \mu}),$}  we find $\phi({\bf x,\mu})$,
such that $||\phi||_{*}$  is $C^1$ in $({\bf x,\mu})$ and \eqref{j3}  holds.
In fact, we will use the contraction mapping theorem to prove the
following result.

\begin{proposition}\label{pro:1} Under the assumptions of Theorem 1.1. If $L>0$
 is sufficiently large,  \eqref{j3} admits a unique solution $\phi=\phi({\bf x, \mu})$ in $\mathbf{H}_n$ such that
$
||\phi||_{*}\leq \frac{C}{\bar \mu^{\min(\frac{N+2m}{2}-\tau, \beta-\tau+1)}},
$
where $\bar \mu =\min (\mu_1, \cdots, \mu_n)$.  Moreover $||\phi||_{*}$ is  $C^1$ in $({\bf x, \mu}).$
\end{proposition}

 Firstly, we consider the following linear problem:
\begin{equation}\label{1-8-1}
(-\Delta)^m \phi
-(m^*-1)K(y)W_n^{m^*-2}\phi=h+\sum_{i=1}^n\sum_{j=1}^{N+1}c_{ij} U_i^{m^*-2} Z_{i,j},
\end{equation}
for some constants $c_{ij}$, where $h$ is a function in $\mathbf{Y}$.

\vskip8pt

\begin{lemma}\label{l1-8-1}
 Suppose that  $\phi$ solves   \eqref{1-8-1}.  Then $\|\phi\|_{*}\le C \|h\|_{**}$, for some
 constant $C>0$, independent of $n$.
\end{lemma}

\begin{proof}
We can write
\begin{equation}\label{2-8-1}
\begin{split}
\phi(y)=&
(m^*-1)\int_{\mathbb R^N}\frac{C_m}{|z-y|^{N-2m}} K(z)W_n^{m^*-2}(z)\phi(z)\,dz\\
&+\int_{\mathbb R^N}\frac{C_m}{|z-y|^{N-2m}}\Bigl(h(z)+\sum_{i=1}^n\sum_{j=1}^{N+1}c_{ij} U_i^{m^*-2}(z) Z_{i,j}(z)\Bigr)dz.
\end{split}
\end{equation}

 Using Lemma~\ref{lem:2}, we have
\begin{equation}\label{2.82}
\begin{split}
\quad \int_{\mathbb R^N}\frac{C_m |h(z)|}{|z-y|^{N-2m}}\,dz&\le C\|h\|_{**} \int_{\mathbb{R}^N}\frac{\sigma(z)}{|y-z|^{N-2m}}\sum\limits_{j}\frac{
\mu_j^{ \frac{N+2m}2 }}{(1+\mu_j|z-x_j|)^{\frac{N+2m}{2}+\tau}}dz
\\ &
\leq C\|h\|_{**}\sigma(y)\sum\limits_{j}\frac{\mu_j^{ \frac{N-2m}2 }}{(1+\mu_j|y-x_j|)^{\frac{N-2m}{2}+\tau}},
\end{split}
\end{equation}
and
\begin{equation}\label{2.83}
\begin{aligned}
\quad\Bigl|\int_{\mathbb{R}^N}\frac{U_i^{m^*-2}Z_{i,j}}{|y-z|^{N-2m}}dz\Bigr|&\leq C\int_{\mathbb{R}^N}\frac{1}{|y-z|^{N-2m}}\frac{\mu_i^{\alpha(j)+ \frac{N+2m}2}}{(1+\mu_i |z-x_i|)^{N+2m}}dz\\
&\leq \frac{C\mu_i^{\alpha(j)+ \frac{N-2m}2}}{(1+\mu_i|y-x_i|)^{\frac{N-2m}{2}+\tau}},
\end{aligned}
\end{equation}
where $\alpha(j)=1$, $j=1,\cdots, N$, $\alpha(N+1)=-1$.

\vskip8pt

 To estimate $c_{ij}$, we use \eqref{1-8-1}  to find
\begin{equation}\label{10-14-1}
-c_{ij} \int_{\mathbb R^N} U_i^{m^*-2} Z^2_{ij}=
\int_{\mathbb R^N}\bigl((m^*-1)K(y)W_n^{m^*-2} \phi +h\bigr) Z_{i,j}.
\end{equation}
But
\begin{equation}\label{11-14-1}
\Bigl|
\int_{\mathbb R^N} h Z_{i,j}\Bigr|\le C\mu_i^{\alpha(i)}\|h\|_{**},
\end{equation}
and (see \eqref{7-2-9})
\begin{equation}\label{12-14-1}
\begin{split}
&\int_{\mathbb R^N}K(y)W_n^{m^*-2} \phi  Z_{i,j}\\
=&\int_{\mathbb R^N}K(y)\bigl(W_n^{m^*-2}- U_i^{m^*-2}\bigr) \phi  Z_{i,j}+\int_{\mathbb R^N}\bigl(K(y)-1) U_i^{m^*-2} \phi  Z_{i,j}+\int_{\mathbb R^N} U_i^{m^*-2} \phi  Z_{i,j}\\
=& \mu_i^{\alpha(i)} \|\phi\|_*O\Bigl( \frac1{\mu_i^{\min(\frac{N+2m}2-\tau, \beta)}} +
\frac1{\mu_i^{N-2m}L^{N-2m}} \Bigr),
\end{split}
\end{equation}
which, together with \eqref{10-14-1} and \eqref{11-14-1}, gives
\begin{equation}\label{3-8-1}
|c_{ij}|\leq C\bigl(||h||_{**}+o(1)||\phi||_{*}\bigr)\mu_i^{-\alpha(j)}.
\end{equation}
Combining Lemma~\ref{lem:4},  and \eqref{2.82}--\eqref{3-8-1}, we are led to
\begin{equation}\label{2.93}
\begin{aligned}
&\quad|\phi(y)|\Bigl(\sigma(y)\sum\limits_{i}
\frac{\mu_i^{\frac{N-2m}2}}{(1+\mu_i|y-x_i|)^{\frac{N-2m}{2}+\tau}}\Bigr)^{-1}\\
&\leq
C\Bigl(||h||_{**}+o(1)||\phi||_{*}
+\frac{ \sum\limits_{j}\frac{1}{(1+\mu_j|z-x_j|)
^{\frac{N-2m}{2}+\tau+\tilde \theta}}  }
{ \sum\limits_{j}\frac{1}{(1+\mu_j|z-x_j|)
^{\frac{N-2m}{2}+\tau}} }  ||\phi||_{*}\Bigr).
\end{aligned}
\end{equation}

 We can finish the proof of this lemma by using \eqref{2.93} as in \cite{WY}.
\end{proof}

\begin{proof}[Proof of Proposition~\ref{pro:1}   ]

Let $P$ be the operator defined as follows:
\[
P f= f+\sum_{i=1}^n\sum_{j=1}^{N+1}c_{ij} U_i^{m^*-2} Z_{i,j},\quad f\in \mathbf{Y},
\]
where $c_{ij}$ are chosen such that $ \int_{\mathbb R^N} Z_{i,j}Pf =0$. Then it is easy to check that
\[
\| P f\|_{**}\le C\| f\|_{**}.
\]
 In view of Lemma~\ref{l1-8-1}, by the Fredholm alternative thoerem, for any $h\in \mathbf Y$, \eqref{1-8-1}
 has a
  unique solution $A h\in \mathbf H_n$.

 Equation~\eqref{j3} is equivalent to
\begin{equation}\label{4-9-1}
\phi = A[P( N(\phi))]+A[P l_L], \quad \phi\in \mathbf H_n,
\end{equation}
where
\begin{equation}\label{2-9-1}
N(\phi)=K(y)\Bigl(\bigl(  W_{n} +\phi\bigr)_+^{m^*-1}-
 W_{n}^{m^*-1}-(m^*-1) W_{n}^{m^*-2}\phi\Bigr),
\end{equation}
and
\begin{equation}\label{3-9-1}
l_L=K(y)
  W_{n}^{m^*-1}-\sum_{j=1}^n U_{j}^{m^*-1}.
\end{equation}
Then, we can use the contraction mapping theorem as in \cite{WY} to prove that for large $L>0$,
\eqref{4-9-1}
has a solution $\phi\in \mathbf H_n$, satisfying
\[
\|\phi\|_{*}\le C \|l_L\|_{**}.
\]
Using Lemma~\ref{l-2-5-3}, we obtain the estimate for $\|\phi\|_{*}$.
\end{proof}

 {\bf Part II: The Finite Dimensional Problems.}   Note that for any $\gamma>1$, we have
$(1+t)_+^\gamma-1-\gamma t = O(t^2)$  for all $t\in \R$ if $\gamma\le 2$; and
$|(1+t)_+^\gamma-1-\gamma t |\le C(t^2+|t|^\gamma)$ for all $t\in\R$ if $\gamma> 2$.
So, we can deduce
\begin{equation}\label{1-15-1}
\begin{split}
&\int_{\mathbb R^N} (-\Delta)^m (W_n(x)+\phi(x)) Z_{i,j}-\int_{\mathbb R^N} K(y) (W_n(x)+\phi(x))_+^{m^*-1}
 Z_{i,j}
\\
=&\int_{\mathbb R^N} (-\Delta)^m W_n(x)  Z_{i,j}-\int_{\mathbb R^N} K(y) W_n(x)^{m^*-1}
 Z_{i,j}\\
 &+(m^*-1) \int_{\mathbb R^N} K(y) W_n^{m^*-2}Z_{i,j}\phi+\mu_i^{\alpha(j)}O\Bigl( (\mu_i^{1-\tau}\|\phi\|_{*})^2\Bigr).
\end{split}
\end{equation}

 It follows from Lemmas~\ref{lemma:2.4} and \ref{lemma:2.5}, Proposition~\ref{pro:1}, \eqref{12-14-1}  and \eqref{1-15-1}
 that \eqref{nnj3}
is  equivalent to
\begin{equation}\label{j3e}
x_{j}- P_{ j}= O(\frac{1}{\bar \mu^2}),
\end{equation}
and
\begin{equation}\label{j4e}
\sum_{i\ne j }\frac{C_4}{(\mu_i\mu_j)^{\frac{N-2m}{2}} |P_i-P_{ j}|^{N-2m}}-\frac{C_3}{\mu_j^\beta}=O(\frac{1}{\bar \mu^{\beta+1}}),\quad
j=1, \cdots, n,
\end{equation}
where $\bar \mu=\min_i\mu_i$.

\subsection{Proof of the existence  theorems}

\begin{proof}
[Proof of Theorem~\ref{th2-7-1}]  We need to solve \eqref{j3e} and \eqref{j4e}.
Note that
\[
\frac{1}{ |P_i-P_j|^{N-2m}}=\frac{1}{ |\tilde P_{ i}-\tilde P_{ j}|^{N-2m} L^{N-2m}}:=
\frac{d_{ij}}{  L^{N-2m}},\quad i\ne j.
\]
So, we can use Lemma~\ref{l1-13-9} to obtain the result.

\end{proof}

\begin{proof}
 [Proof of Theorem~\ref{th11}]  The proof of Theorem~\ref{th11} follows from Theorem~\ref{th2-7-1} by a limiting argument, since we can easily shown that for any fixed $L>0$ large, there exists some constant $C=C(L)$, independent of $n$,  such that
\begin{equation}\label{theorem31}
{u}_n(x)\leq C(L),\quad  \forall\; x\in \mathbb{R}^N.
\end{equation}
 By elliptic estimate, for any $R>0,$ there exists a constant $C_2=C_2(L)$ independent of $n,$ such that
$\| u_n(x)\|_{C^{2m}(B_R)}\leq C_2(L)$, $\forall\; n=1,\cdots,$
which implies that (up to a subsequence, still denoted by ${u}_n$)
$u_n\to u \quad \hbox{ in  } C_{loc}^{2m}(\mathbb{R}^N),$
satisfying
$(-\Delta )^m u=K(x) u_+^{\frac{N+2m}{N-2m}}  \hbox{ in } \mathbb{R}^N.$  Noticing that $u$ decays at direction $y_N$, we can deduce from the potential theory for elliptic equations that
$$
u(x)=\int_{\R^N}\frac{K(y)}{|y-x|^{N-2m}}u_+(y)^{\frac{N+2m}{N-2m}},
$$
which also implies  $u>0$.

\end{proof}

\section{Local uniqueness and periodicity}

 In this section, we study the local uniqueness of the bubbling solutions for $(P)$. We assume that conditions
$(A_1)$--$(A_3)$ hold.  Suppose that $u^{(1)}_{L}
$  and
 $u^{(2)}_{L}$ are two sequence of solutions of $(P)$, which
 satisfy \eqref{3-6-1}, \eqref{2-6-1}  and \eqref{2-1-9}.
 We will prove that
$u^{(1)}_{ L}= u^{(2)}_{ L}$ provided $L>0$ is large enough.

%

\subsection{Pohozaev type identities}

 Suppose that $u$ and $v$ are two  smooth functions in a given bounded domain $\Omega$.  In the section,
we study the following two bi-linear functionals:
\begin{eqnarray}
 L_{1, i}(u, v)&=& \int_\Omega \bigl( (-\Delta)^m u \frac{\partial v}{\partial y_i}+ (-\Delta)^m v \frac{\partial u}{\partial y_i}\bigr), \quad i=1, 2, ..., N, \label{1-13-10}\\
 L_2(u, v)&=& \int_\Omega \bigl( (-\Delta)^m u \bigl\langle y-x, \nabla v\bigr\rangle+ (-\Delta)^m  v
 \bigl\langle y-x, \nabla u\bigr\rangle\bigr). \label{2-13-10}
\end{eqnarray}

\begin{proposition}\label{p1-13-10}
For any integer $m>0$, there exists a function $f_{m, i}(u, v)$,
   such that
\begin{equation}\label{3-13-10}
 L_{1, i}(u, v)= \int_{\partial \Omega } f_{m, i}(u, v).
\end{equation}
Moreover, $f_{m, i}(u, v)$ has the following form:
\begin{equation}\label{300-13-10}
f_{m, i}(u, v)= \sum_{j=1}^{2m-1} l_{j,i}( \nabla^j u,  \nabla^{2m-j }v),
\end{equation}
where $l_{j,i}( \nabla^j u,  \nabla^{2m-j }v)$ is bi-linear in $\nabla^j u$ and $\nabla^{2m-j} v$.
\end{proposition}

\begin{proof}
For $m=1$, we use the integration by parts to find
\begin{equation}\label{4-13-10}
\begin{split}
 L_{1, i}(u, v)= &-\int_{\partial \Omega } \bigl(\frac{\partial u}{\partial \nu}\frac{\partial v}{\partial y_i}+
 \frac{\partial v}{\partial \nu}\frac{\partial u}{\partial y_i}\bigr)+ \int_{ \Omega } \bigl(\frac{\partial u}{\partial y_j}\frac{\partial^2 v}{\partial y_j \partial y_i}+
 \frac{\partial v}{\partial y_j}\frac{\partial^2 u}{\partial y_j \partial y_i}\bigr)\\
 =&-\int_{\partial \Omega } \bigl(\frac{\partial u}{\partial \nu}\frac{\partial v}{\partial y_i}+
 \frac{\partial v}{\partial \nu}\frac{\partial u}{\partial y_i}\bigr)+
 \int_{\partial \Omega }\bigl\langle \nabla u, \nabla v\bigr\rangle \nu_i.
 \end{split}
\end{equation}

 For any integer $m>1$, we have
\begin{equation}\label{5-13-10}
\begin{split}
& \int_\Omega \bigl( (-\Delta)^m  u\frac{\partial v}{\partial y_i}+ (-\Delta)^m  v\frac{\partial u}{\partial y_i}\bigr)\\
=& -\int_{\partial \Omega } \bigl(\frac{\partial (-\Delta)^{m-1} u}{\partial \nu}\frac{\partial v}{\partial y_i}+
 \frac{\partial  (-\Delta)^{m-1} v}{\partial \nu}\frac{\partial u}{\partial y_i}\bigr)\\
 &+ \int_{ \Omega } \bigl(\frac{\partial  (-\Delta)^{m-1} u}{\partial y_j}\frac{\partial^2 v}{\partial y_j \partial y_i}+
 \frac{\partial  (-\Delta)^{m-1} v}{\partial y_j}\frac{\partial^2 u}{\partial y_j \partial y_i}\bigr)\\
 =&-\int_{\partial \Omega } \bigl(\frac{\partial (-\Delta)^{m-1} u}{\partial \nu}\frac{\partial v}{\partial y_i}+
 \frac{\partial  (-\Delta)^{m-1} v}{\partial \nu}\frac{\partial u}{\partial y_i}\bigr)\\
 &+ \int_{ \partial\Omega } \bigl( (-\Delta)^{m-1} u\bigl\langle \nu , \nabla\frac{\partial v}{\partial y_i}
 \bigr\rangle +
 (-\Delta)^{m-1} v\bigl\langle \nu , \nabla\frac{\partial u}{\partial y_i}
 \bigr\rangle\bigr) \\
 &-\int_{ \Omega } \bigl( (-\Delta)^{m-1} u\frac{\partial\Delta  v}{ \partial y_i}+
 (-\Delta)^{m-1} v\frac{\partial\Delta  u}{ \partial y_i}\bigr).
  \end{split}
\end{equation}
If $m=2$, then the last term in \eqref{5-13-10} becomes
\[
-\int_{ \Omega } \bigl( (-\Delta u)\frac{\partial\Delta  v}{ \partial y_i}+
 (-\Delta v)\frac{\partial\Delta  u}{ \partial y_i}\bigr)=\int_{\partial \Omega } \Delta u \Delta v \nu_i,
 \]
which gives
\begin{equation}\label{n5-13-10}
\begin{split}
& \int_\Omega \bigl( (-\Delta)^2 u\frac{\partial v}{\partial y_i}+ (-\Delta)^2  v\frac{\partial u}{\partial y_i}\bigr)\\
 =&-\int_{\partial \Omega } \bigl(\frac{\partial (-\Delta u)}{\partial \nu}\frac{\partial v}{\partial y_i}+
 \frac{\partial  (-\Delta v)}{\partial \nu}\frac{\partial u}{\partial y_i}\bigr)\\
 &+ \int_{ \partial\Omega } \bigl( (-\Delta u)\bigl\langle \nu , \nabla\frac{\partial v}{\partial y_i}
 \bigr\rangle +
 (-\Delta v)\bigl\langle \nu , \nabla\frac{\partial u}{\partial y_i}
 \bigr\rangle\bigr) +\int_{\partial \Omega } \Delta u \Delta v \nu_i.
  \end{split}
\end{equation}

 We assume that the conclusion holds up to  $m-1$, $m\ge 3$.
Take $u_1=-\Delta u$
   and $v_1=-\Delta v$. Then
 \begin{equation}\label{6-13-10}
-\int_{ \Omega } \bigl( (-\Delta)^{m-1} u\frac{\partial\Delta  v}{ \partial y_i}+
 (-\Delta)^{m-1} v\frac{\partial\Delta  u}{ \partial y_i}\bigr)
  =\int_{ \Omega } \bigl( (-\Delta)^{m-2} u_1\frac{\partial v_1}{ \partial y_i}+
 (-\Delta)^{m-2} v_1\frac{\partial u_1}{ \partial y_i}\bigr).
\end{equation}
Thus, by using the induction assumption, we can conclude  that the result is true for any $m.$
\end{proof}

\begin{remark}\label{re1-14-10}
 From the the proof of Proposition~\ref{p1-13-10}, we can see that if $\Omega$ is a ball centered at
 $x$,  $u$ and $v$ are functions of $|y-x|$, then there exists a function $\tilde f_m(r)$, such that
$ f_{m, i}(u, v)= \tilde f_m(|y-x|) \nu_i. $ As a  result, $\int_{\partial B_d(x)}f_{m, i}(u, v)=0.$
\end{remark}

\begin{proposition}\label{p2-13-10}
 For any integer $m>0$, there exists a function $g_m(u, v)$,
   such that
\begin{equation}\label{7-13-10}
 L_2(u, v)= \int_{\partial \Omega } g_m(u, v) - \frac{N-2m}2\int_{\Omega} \bigl( v (-\Delta )^m u + u (-\Delta )^m v  \bigr).
\end{equation}
Moreover, $g_{m}(u, v)$ has the following form:
\begin{equation}\label{301-13-10}
g_{m}(u, v)= \sum_{j=1}^{2m-1} \bar l_j( y-x, \nabla^j u,  \nabla^{2m-j }v)+ \sum_{j=0}^{2m-1} \tilde l_j( \nabla^j u,  \nabla^{2m-j -1}v),
\end{equation}
where $\bar l_j( y-x, \nabla^j u,  \nabla^{2m-j }v)$ and
$ \tilde l_j( \nabla^j u,  \nabla^{2m-j -1}v)  $ are  linear in each component.
\end{proposition}

\begin{proof}

If $m=1$, then using integration by parts, we obtain
\begin{equation}\label{8-13-10}
  \begin{split}
 & \int_\Omega \bigl( (-\Delta u) \bigl\langle y-x, \nabla v\bigr\rangle+ (-\Delta  v)
 \bigl\langle y-x, \nabla u\bigr\rangle\bigr)\\
 =& -\int_{\partial \Omega }\bigl( \frac{\partial u}{\partial \nu}\bigl\langle y-x, \nabla v\bigr\rangle+
 \frac{\partial v}{\partial \nu}\bigl\langle y-x, \nabla u\bigr\rangle\bigr)
 +\int_{\partial \Omega }\langle y-x, \nu\rangle \langle\nabla u, \nabla v\rangle\\
 &-\frac{N-2}2\int_{\partial \Omega }\bigl(\frac{\partial u}{\partial \nu}v+\frac{\partial v}{\partial \nu}u\bigr)
 +\frac{N-2}2\int_\Omega \bigl( v\Delta u + u\Delta v\bigr).
 \end{split}
\end{equation}
For any integer $m>1$, we have
\begin{equation}\label{1-14-10}
  \begin{split}
& \int_\Omega \bigl( (-\Delta)^m u \bigl\langle y-x, \nabla v\bigr\rangle+ (-\Delta)^m  v
 \bigl\langle y-x, \nabla u\bigr\rangle\bigr)\\
 =&-\int_{\partial \Omega }\bigl( \frac{\partial (-\Delta)^{m-1} u}{\partial \nu}\bigl\langle y-x, \nabla v\bigr\rangle+ \frac{\partial (-\Delta)^{m-1} v}{\partial \nu}\bigl\langle y-x, \nabla u\bigr\rangle\bigr)\\
 &+\int_\Omega \bigl( \frac{\partial( (-\Delta)^{m-1} u)}{\partial y_j}
 \frac{\partial v}{\partial y_j}+ \frac{\partial( (-\Delta)^{m-1} v)}{\partial y_j}
 \frac{\partial u}{\partial y_j}
 \bigr)\\
 &+\int_\Omega \bigl( \frac{\partial( (-\Delta)^{m-1} u)}{\partial y_j}   \bigl\langle y-x, \nabla\frac{\partial v}{\partial y_j} \bigr\rangle+
 \frac{\partial( (-\Delta)^{m-1} v)}{\partial y_j}   \bigl\langle y-x, \nabla\frac{\partial u}{\partial y_j} \bigr\rangle
 \bigr)\\
  =&-\int_{\partial \Omega }\bigl( \frac{\partial (-\Delta)^{m-1} u}{\partial \nu}\bigl\langle y-x, \nabla v\bigr\rangle+ \frac{\partial (-\Delta)^{m-1} v}{\partial \nu}\bigl\langle y-x, \nabla u\bigr\rangle\bigr)\\
  &+\int_{\partial \Omega } \bigl( (-\Delta)^{m-1} u  \bigl\langle y-x, \nabla\frac{\partial v}{\partial \nu} \bigr\rangle+
  (-\Delta)^{m-1} v   \bigl\langle y-x, \nabla\frac{\partial u}{\partial \nu} \bigr\rangle
 \bigr)\\
  &+\int_{\partial \Omega } \bigl(  (-\Delta)^{m-1} u
 \frac{\partial v}{\partial \nu}+ (-\Delta)^{m-1} v
 \frac{\partial u}{\partial \nu}
 \bigr)\\
 &- 2\int_\Omega \bigl(  (-\Delta)^{m-1} u
 \Delta  v+ (-\Delta)^{m-1} v
 \Delta  u
 \bigr)\\
 &-\int_\Omega \bigl(  (-\Delta)^{m-1} u  \bigl\langle y-x, \nabla\Delta  v\bigr\rangle+
 (-\Delta)^{m-1} v  \bigl\langle y-x, \nabla\Delta  u\bigr\rangle
 \bigr).
\end{split}
\end{equation}
Letting $m=2$ in \eqref{1-14-10}, we obtain
\begin{equation}\label{3-14-10}
  \begin{split}
& \int_\Omega \bigl( (-\Delta)^2 u \bigl\langle y-x, \nabla v\bigr\rangle+ (-\Delta)^2  v
 \bigl\langle y-x, \nabla u\bigr\rangle\bigr)\\
  =&-\int_{\partial \Omega }\bigl( \frac{\partial (-\Delta u)}{\partial \nu}\bigl\langle y-x, \nabla v\bigr\rangle+ \frac{\partial (-\Delta v)}{\partial \nu}\bigl\langle y-x, \nabla u\bigr\rangle\bigr)\\
  &+\int_{\partial \Omega } \bigl( (-\Delta u)  \bigl\langle y-x, \nabla\frac{\partial v}{\partial \nu} \bigr\rangle+
  (-\Delta v )  \bigl\langle y-x, \nabla\frac{\partial u}{\partial \nu} \bigr\rangle
 \bigr)\\
  &+\int_{\partial \Omega } \bigl(  (-\Delta u)
 \frac{\partial v}{\partial \nu}+ (-\Delta v)
 \frac{\partial u}{\partial \nu}
 \bigr)\\
 &
 +
 \int_{\partial \Omega }  \bigl\langle y-x, \nu\bigr\rangle\Delta u\Delta  v
-\frac{N-4}2\int_{\partial \Omega } \bigl( \frac{\partial u}{\partial \nu}\Delta  v-u \frac{\partial \Delta v}{\partial \nu}+\frac{\partial v}{\partial \nu}\Delta  u-v \frac{\partial \Delta u}{\partial \nu}
 \bigr)
 \\
 &-\frac{N-4}2\int_\Omega\bigl(v (-\Delta)^2 u+
 u(-\Delta )^2 v\bigr),
\end{split}
\end{equation}
since
\[
 \begin{split}
&4\int_\Omega \Delta u
 \Delta  v
 +
 \int_\Omega   \bigl\langle y-x, \nabla(\Delta u\Delta  v)\bigr\rangle=
 \int_{\partial \Omega }  \bigl\langle y-x, \nu\bigr\rangle\Delta u\Delta  v-(N-4)\int_\Omega \Delta u
 \Delta  v\\
 =&\int_{\partial \Omega }  \bigl\langle y-x, \nu\bigr\rangle\Delta u\Delta  v
 -\frac{N-4}2\int_{\partial \Omega } \bigl( \frac{\partial u}{\partial \nu}\Delta  v-u \frac{\partial \Delta v}{\partial \nu}+\frac{\partial v}{\partial \nu}\Delta  u-v \frac{\partial \Delta u}{\partial \nu}
 \bigr)\\
 &-\frac{N-4}2\int_\Omega\bigl(v (-\Delta)^2 u+
 u(-\Delta )^2 v\bigr).
\end{split}
\]

 For any integer $m>2$, we assume that the result is true for any integer up to $m-1$.
 First, we have
\begin{equation}\label{4-14-10}
  \begin{split}
  &- 2\int_\Omega \bigl(  (-\Delta)^{m-1} u
 \Delta  v+ (-\Delta)^{m-1} v
 \Delta  u
 \bigr)\\
 =&-2\int_{\partial \Omega } \bigl(
 (-\Delta)^{m-1} u\frac{\partial v}{\partial \nu}-(-\Delta)^{m-1}\frac{\partial u}{\partial \nu} v+
 (-\Delta)^{m-1} v\frac{\partial u}{\partial \nu}-(-\Delta)^{m-1}\frac{\partial v}{\partial \nu} u
 \bigr)
 \\
 &
 + 2\int_\Omega \bigl(  (-\Delta)^{m} u)
  v+ (-\Delta)^{m} v)
   u
 \bigr).
\end{split}
\end{equation}
Moreover, by the induction assumption, we obtain by using the integration by parts,
\begin{equation}\label{5-14-10}
  \begin{split}
  &-\int_\Omega \bigl(  (-\Delta)^{m-1} u  \bigl\langle y-x, \nabla\Delta  v\bigr\rangle+
 (-\Delta)^{m-1} v  \bigl\langle y-x, \nabla\Delta  u\bigr\rangle
 \bigr)\\
 =&\int_{\partial \Omega } g_{m-2}(-\Delta u, -\Delta
 v) - \frac{N-2(m-2)}2\int_{\Omega} \bigl( (-\Delta v) (-\Delta )^{m-1} u + (-\Delta u) (-\Delta )^{m-1} v  \bigr)\\
 =&
 \int_{\partial \Omega } g_{m}( u,
 v) - \frac{N-2(m-2)}2\int_{\Omega} \bigl(  v (-\Delta )^{m} u +  u (-\Delta )^{m} v  \bigr).
\end{split}
\end{equation}
Hence, the result for any $m$ follows from \eqref{1-14-10}, \eqref{4-14-10} and \eqref{5-14-10}.
\end{proof}

\begin{remark}\label{re1-15-10}
 From \eqref{4-14-10} and \eqref{5-14-10}, we can find the formula for $\tilde l_{2m-1}( \nabla^{2m-1} u,  v) $:
 \begin{equation}\label{88-14-10}
 \tilde l_{2m-1}( \nabla^{2m-1} u,  v)=-\frac{N-2m}2 \int_{\partial \Omega }\frac{\partial (-\Delta )^{m-1} u}{\partial
 \nu} v.
 \end{equation}
 \end{remark}

\subsection{The bubbling solutions}

Let $u_{L}=\sum_{i=1}^{\infty }U_{x_{i, L}, \mu_{i, L}}+\omega_L
$  be a  solution of $(P)$, which
 satisfies \eqref{3-6-1}, \eqref{2-6-1}  and \eqref{2-1-9}. In this section, we will estimate $\mu_{j, L}$ and $
 |x_{j, L}- P_j|$. We will use various Pohozaev identities to achieve this.

\vskip8pt

  Using Propositions~\ref{p1-13-10} and \ref{p2-13-10},  we can obtain the  following two Pohozaev identities:

 \begin{equation}\label{1-4-9}
\frac12 \int_{\partial B_\delta(x_{j, L})}f_{m,i}(u_L, u_L)= \frac1{m^*}
\int_{\partial B_\delta(x_{j, L})}K(y) u_{L}^{m^*} \nu_i-
\frac1{m^*}
\int_{ B_\delta(x_{j, L})}\frac{\partial K(y)}{\partial y_i} u_{L}^{m^*},
\end{equation}
and

\begin{equation}\label{2-4-9}
\begin{split}
&\frac12 \int_{\partial B_\delta(x_{j, L})}g_m(u_L, u_L)\\
 = & \frac1{m^*}
\int_{\partial B_\delta(x_{j, L})}K(y) u_{L}^{m^*} \bigl\langle y-x_{j, L}, \nu\bigr\rangle-
\frac1{m^*}
\int_{ B_\delta(x_{j, L})} \bigl\langle \nabla K(y),  y-x_{j, L}\bigr\rangle u_{L}^{m^*},
\end{split}
\end{equation}
where $\nu$ is the outward unit normal of $\partial B_\delta(x_{j, L})$. We will estimate each term in \eqref{1-4-9} and \eqref{2-4-9}.

 We denote $\mu_L=\max_j \mu_{j, L}$. Note that for $y\in \partial B_\delta(x_{j, L})$,
\begin{equation}\label{10-14-10}
U_{x_{j,L}, \mu_{j,L}}=\frac{\tilde C_m}{\mu_{j, L}^{\frac{N-2m}2}} \sum_{i=0}^{m-1}
\frac{\alpha_i}{\mu_{j, L}^{2i}|y-x_{j, L}|^{N-2m+2i}}+O\bigl(\frac1{\mu_L^{\frac{N+2m}2}}\bigr),
\end{equation}
where $\alpha_0=1$, and $\alpha_i\ne 0$ is a constant, $i=1, \cdots, m-1$.

Let
\begin{equation}\label{12-14-10}
V_{j, L}(y)= \frac{\tilde C_m}{\mu_{j, L}^{\frac{N-2m}2}} \sum_{i=0}^{m-1}
\frac{\alpha_i}{\mu_{j, L}^{2i}|y-x_{j, L}|^{N-2m+2i}}.
\end{equation}
Then, we have
\begin{equation}\label{2.4}
\begin{split}
u_L(y)=&V_{j, L}(y)+ O\Bigl( \frac1{\mu_L^{\frac{N+2m}2}}+\frac1{\mu_L^{\frac{N-2m}2} L^{N-2m}}+\frac{\|\omega_L\|_{*}    }{\mu_L^\tau}\Bigr),\quad y\in \partial B_\delta(x_{j, L}).
\end{split}
\end{equation}

Since
\[
(-\Delta )^{m-i} \frac1{|y-x_{j, L}|^{N-2(m-i)}}=0,\;\;\; \text{in}\; \mathbb R^N \setminus \{x_{j, L}\},
\]
we have
\begin{equation}\label{11-14-10}
(-\Delta )^{m}
\frac{\alpha_i}{\mu_{j, L}^{2i}|y-x_{j, L}|^{N-2m+2i}}=0,\quad \hbox{ in } \mathbb R^N \setminus \{x_{j, L}\}.
\end{equation}
So on $\partial B_\delta(x_{j, L})$, the leading term $V_{j, L}(y)$
of $u_L$ is an $m$-harmonic function. From
this observation, by using Propositions~\ref{p1-13-10} and \ref{p2-13-10}, the estimates of the surface integrals
 on $\partial B_\delta(x_{j, L})$  in the left hand side of \eqref{1-4-9} and \eqref{2-4-9}
for $V_{j, L}(y)$ can be reduced to the estimates
of the surface integrals
  on $\partial B_\theta(x_{j, L})$ for any small number $\theta>0$, which can be done because
  we know the singular behavior  at $x_{j, L}$ of the function $V_{j, L}(y)$.

\begin{lemma}\label{l1-26-9}

Relation~\eqref{1-4-9} is equivalent to
\begin{equation}\label{4-8-9}
 \begin{split}
\int_{ B_\delta(x_{j, L})}\frac{\partial K(y)}{\partial y_i}
U_{x_{j,L}, \mu_{j,L}}^{m^*}=&O\Bigl(\frac{1}{\mu_L^N}+\frac1{\mu_L^{N-2m}L^{N-2m}}+\frac{\max_i|x_{i,L}-P_i|^\beta}{\mu_L^{\frac{N-2m}{2}+\tau}}\\
&+\frac{\max_i|x_{i,L}-P_i|^{2\beta}}{\mu_L^{2\tau}}
+\max_i|x_{i,L}-P_i|^{m^*\beta}
\Bigr).
\end{split}
\end{equation}
\end{lemma}

\begin{proof}
It follows from
\eqref{2.4} that
\begin{equation}\label{1-8-9}
 \int_{\partial B_\delta(x_{j, L})}K(y) u_{L}^{m^*}
\nu_i=O\bigl(\frac1{\mu_L^N}  \bigr).
\end{equation}

To estimate the left hand side of \eqref{1-4-9}, noting that $V_{j, L}(y)$ is a function of $|y-x_{j, L}|$,
 we use Remark~\ref{re1-14-10} to obtain
\begin{equation}\label{2.7}
\begin{split}
&\text{LHS of \eqref{1-4-9}}\\
=&\frac12 \int_{\partial B_\delta(x_{j, L})}f_{m,i}(V_{j, L}, V_{j, L})+
O\Bigl( \frac1{\mu_L^{N}}+\frac1{\mu_L^{N-2m} L^{N-2m}}+\frac{\|\omega_L\|_{*}    }{\mu_L^{\frac{N-2m}2+\tau}}
+\frac{\|\omega_L\|^2_{*}    }{\mu_L^{2\tau}}\Bigr)\\
=&O\Bigl( \frac1{\mu_L^{N}}+\frac1{\mu_L^{N-2m} L^{N-2m}}+\frac{\|\omega_L\|_{*}    }{\mu_L^{\frac{N-2m}2+\tau}}
+\frac{\|\omega_L\|^2_{*}    }{\mu_L^{2\tau}}\Bigr).
\end{split}
\end{equation}
Moreover, from
\[
\int_{ B_\delta(x_{j, L})}\Bigl( U_{x_{j,L}, \mu_{j,L}}^{m^*-1} \sum_{i\ne j} U_{x_{j,L}, \mu_{j,L}}
+ \bigl(\sum_{i\ne j} U_{x_{j,L}, \mu_{j,L}}\bigr)^{m^*}\Bigr)= O\bigl(\frac1{\mu_L^{N-2m}L^{N-2m}}+\frac1{\mu_L^{N}}
\bigr),
\]
we find
\begin{equation}\label{3-8-9}
\begin{split}
\int_{ B_\delta(x_{j, L})}\frac{\partial K(y)}{\partial y_i}
u_{L}^{m^*}=&\int_{ B_\delta(x_{j, L})}\frac{\partial K(y)}{\partial
y_i} U_{x_{j,L}, \mu_{j,L}}^{m^*}+O\Bigl(\frac1{\mu_L^{N-2m}L^{N-2m}}+\frac1{\mu_L^{N}}\Bigr)\\
&
+O\Bigl(\|\omega_L\|_{*}\bigl(\frac1{\mu_L^{\frac{N+2m}2}}+|x_{j,L}-P_j|^\beta
\bigr)+
\|\omega_L\|_{*}^{m^*}\Bigr).
\end{split}
\end{equation}
Thus, \eqref{4-8-9} follows  from \eqref{1-8-9}, \eqref{2.7} and \eqref{3-8-9}.
\end{proof}

\begin{lemma}\label{l2-26-9}
Relation~\eqref{2-4-9} is equivalent to
\begin{equation}\label{10-10-9}
\begin{split}
 & \frac {1}{\mu_{j, L}^{\beta}}-B\sum_{i\ne j} \frac1{\mu_{i, L}^{\frac{N-2m}2}\mu_{j, L}^{\frac{N-2m}2}
|x_{i, L}-x_{j, L}|^{N-2m}}\\
=&
O\Bigl(\frac1{\mu_L^{N}}+\frac1{\mu_L^{N-2m+2} L^{N-2m}}+\frac1{\mu_L^{N-2m} L^{2(N-2m)}}+\frac1{\mu_{L}^{\beta+1}}\\
&+
\frac{\max_i |x_{i, L}- P_i|^{}}{\mu_L^{\beta-1}}+\frac{\max_i |x_{i, L}- P_i|^{\beta-1}}{\mu_L}
\Bigr),
\end{split}
\end{equation}
where $B>0$ is a constant.
\end{lemma}

\begin{proof}

We first estimate the left hand side of  \eqref{2-4-9}. We have
\begin{equation}\label{80-14-10}
\begin{split}
u_L(y)=&V_{j, L}(y)+\sum_{i\ne j}\frac{\tilde C_m}{\mu_{i, L}^{\frac{N-2m}2}|y-x_{i, L}|^{N-2m}
}\\
&+ O\Bigl( \frac1{\mu_L^{\frac{N+2m}2}}+\frac1{\mu_L^{\frac{N-2m}2+2} L^{N-2m+2}}+\frac{\|\omega_L\|_{*}    }{\mu_L^\tau}\Bigr),\quad y\in \partial B_\delta(x_{j, L}).
\end{split}
\end{equation}
Let
\begin{equation}\label{81-14-10}
Q_{j, L}(y)=V_{j, L}(y)+\sum_{i\ne j}\frac{\tilde C_m}{\mu_{i, L}^{\frac{N-2m}2}|y-x_{i, L}|^{N-2m}
}=: V_{j, L}(y)+ R_{j, L}(y).
\end{equation}
    We have
\begin{equation}\label{20-14-10}
\begin{split}
&\int_{\partial B_\delta(x_{j, L})}g_m(u_L, u_L)\\
=& \int_{\partial B_\delta(x_{j, L})}g_m(Q_{j, L}, Q_{j, L})+O\Bigl( \frac1{\mu_L^{N}}+\frac1{\mu_L^{N-2m+2} L^{N-2m+2}}+\frac{\|\omega_L\|_{*}    }{\mu_L^{\frac{N-2m}2+\tau}}
+\frac{\|\omega_L\|^2_{*}    }{\mu_L^{2\tau}}\Bigr).
\end{split}
\end{equation}
Note that    $ R_{j, L}(y)$ is bounded on $ \partial B_\delta(x_{j, L})  $.  Now we compute
\begin{equation}\label{1-15-10}
\begin{split}
& \int_{\partial B_\delta(x_{j, L})}g_m(Q_{j, L}, Q_{j, L})\\
=&\int_{\partial B_\delta(x_{j, L})}g_m(V_{j, L}, V_{j, L})+ 2 \int_{\partial B_\delta(x_{j, L})}g_m(V_{j, L},
 R_{j, L})+\int_{\partial B_\delta(x_{j, L})}g_m(R_{j, L},
 R_{j, L})\\
 =&\int_{\partial B_\delta(x_{j, L})}g_m(V_{j, L}, V_{j, L})+ 2 \int_{\partial B_\delta(x_{j, L})}g_m(V_{j, L},
 R_{j, L})+O\bigl(\frac1{\mu_L^{N-2m} L^{2(N-2m)}}\bigr).
\end{split}
\end{equation}
Moreover, from the definition of $V_{j, L}$, we get
\begin{equation}\label{2-15-10}
\begin{split}
 &\int_{\partial B_\delta(x_{j, L})}g_m(V_{j, L},
 R_{j, L})\\
 =& \frac{\tilde C_m}{\mu_{j, L}^{\frac{N-2m}2}}\int_{\partial B_\delta(x_{j, L})}g_m(
\frac{1}{|y-x_{j, L}|^{N-2m}},
 R_{j, L})+ O\bigl(\frac1{\mu_L^{N-2m+2} L^{N-2m}}\bigr).
\end{split}
\end{equation}
Note that both $\frac{1}{|y-x_{j, L}|^{N-2m}}$ and $R_{j, L}$ are $m$-harmonic in $\Omega= B_\delta(x_{j, L})\setminus B_\theta(x_{j, L})$, where
$\theta>0$ is any small constant.
We apply Proposition~\ref{p2-13-10} in $\Omega$ to obtain
\begin{equation}\label{21-14-10}
\begin{split}
\int_{\partial B_\delta(x_{j, L})}g_m(
\frac{1}{|y-x_{j, L}|^{N-2m}},
 R_{j, L})=\int_{\partial B_\theta(x_{j, L})}g_m(
\frac{1}{|y-x_{j, L}|^{N-2m}},
 R_{j, L}).
 \end{split}
\end{equation}
Since $R_{j, L}$ and its derivatives are bounded on $\partial B_\theta(x_{j, L})$,
we find that  the term in  \eqref{301-13-10} satisfies
\[
|\bar l_l( y-x_{j, L}, \nabla^l \frac{1}{|y-x_{j, L}|^{N-2m}},  \nabla^{2m-l }R_{j, L})|\le
\frac{C}{|y-x_{j, L}|^{N-2}}, \quad l=0, \cdots, 2m-1.
\]
As a result, as $\theta\to 0$,
\begin{equation}\label{3-15-10}
 \int_{\partial B_\theta(x_{j, L})}\bar l_l( y-x_{j, L}, \nabla^l \frac{1}{|y-x_{j, L}|^{N-2m}},  \nabla^{2m-l }R_{j, L})\to 0, \quad l=0, \cdots, 2m-1.
\end{equation}
 For $l=0,\cdots,  2m-2$, the other terms in  \eqref{301-13-10} satisfies
\[
 |\tilde l_l( \nabla^l \frac{1}{|y-x_{j, L}|^{N-2m}},  \nabla^{2m-l-1} R_{j, L} )|\le \frac{C}{|y-x_{j, L}|^{N-2}}.
 \]
Therefore,
\begin{equation}\label{4-15-10}
 \int_{\partial B_\theta(x_{j, L})}\tilde l_l( \nabla^l \frac{1}{|y-x_{j, L}|^{N-2m}},  \nabla^{2m-l-1}R_{j, L} )\to 0, \quad l=0, \cdots, 2m-2, \hbox{ as } \theta\to 0.
\end{equation}

Using \eqref{88-14-10}, we have
\begin{equation}\label{nn5-15-10}
\begin{split}
& \int_{\partial B_\theta(x_{j, L})}\tilde l_{2m-1}( \nabla^{2m-1} \frac{1}{|y-x_{j, L}|^{N-2m}},  R_{j, L} )\\
=&-\frac{N-2m}2
 \int_{\partial B_\theta(x_{j, L})}\frac{\partial [(-\Delta)^{m-1}\frac{1}{|y-x_{j, L}|^{N-2m}}]  }{\partial \nu}  R_{j, L}.
 \end{split}
\end{equation}
The function $(-\Delta)^{m-1}\frac{1}{|y-x_{j, L}|^{N-2m}}$ depends on $|y-x_{j, L}|$ only, which satisfies
\[
 -\Delta [(-\Delta)^{m-1}\frac{1}{|y-x_{j, L}|^{N-2m}}]=c_m\delta_{x_{j, L}},
 \]
for some constant $c_m>0$.
As a result,
there is a constant $c'_m>0$, such that
\[
(-\Delta)^{m-1}\frac{1}{|y-x_{j, L}|^{N-2m}}=\frac{c_m'}{|y-x_{j, L}|^{N-2}}.
\]
This, together with \eqref{nn5-15-10} gives
\begin{equation}\label{5-15-10}
\begin{split}
\int_{\partial B_\theta(x_{j, L})}\tilde l_0( \nabla^{2m-1} \frac{1}{|y-x_{j, L}|^{N-2m}},  R_{j, L} )=&\frac{(N-2)(N-2m)}2\frac{c_m'}{\theta^{N-1}}
 \int_{\partial B_\theta(x_{j, L})} R_{j, L}\\
 =& \sum_{i\ne j}\frac{B_m}{\mu_{i, L}^{\frac{N-2m}2}|x_{j, L}-x_{i, L}|^{N-2m}
}\bigl( 1+o_\theta(1)\bigr),
 \end{split}
\end{equation}
where $B_m>0$ is a constant.

 Combining \eqref{2-15-10}--\eqref{5-15-10}, we conclude  the existence of some constant  $B_m'>0$ such that
\begin{equation}\label{6-15-10}
\begin{split}
&2\int_{\partial B_\delta(x_{j, L})}g_m(
V_{j, L},
 R_{j, L})\\
 =& \sum_{i\ne j}\frac{B'_m}{\mu_{j, L}^{\frac{N-2m}2}\mu_{i, L}^{\frac{N-2m}2}|x_{j, L}-x_{i, L}|^{N-2m}
}+ O\bigl(\frac1{\mu_L^{N-2m+2} L^{N-2m}}\bigr).
 \end{split}
\end{equation}

 We are now to estimate
\begin{equation}\label{7-15-10}
\begin{split}
& \int_{\partial B_\delta(x_{j, L})}g_m(V_{j, L}, V_{j, L})\\
=& \frac{\tilde C^2_m}{\mu_{j, L}^{N-2m}} \sum_{h=0}^{m-1} \sum_{k=0}^{m-1}
\frac{\alpha_h\alpha_k}{\mu_{j, L}^{2h}\mu_{j, L}^{2k}}\int_{\partial B_\delta(x_{j, L})}g_m\bigl(
\frac{1}{|y-x_{j, L}|^{N-2m+2h}}, \frac{1}{|y-x_{j, L}|^{N-2m+2k}} \bigr).
\end{split}
\end{equation}
Since $\frac{1}{|y-x_{j, L}|^{N-2m+2h}}$  and $ \frac{1}{|y-x_{j, L}|^{N-2m+2k}}$ are $m$-harmonic
in $B_\delta(x_{j, L})\setminus  B_\theta(x_{j, L})$,
we can use Proposition~\ref{p2-13-10} to obtain
\begin{equation}\label{8-15-10}
\begin{split}
& \int_{\partial B_\delta(x_{j, L})}g_m\bigl(
\frac{1}{|y-x_{j, L}|^{N-2m+2h}}, \frac{1}{|y-x_{j, L}|^{N-2m+2k}} \bigr)\\
=&\int_{\partial B_\theta(x_{j, L})}g_m\bigl(
\frac{1}{|y-x_{j, L}|^{N-2m+2h}}, \frac{1}{|y-x_{j, L}|^{N-2m+2k}} \bigr).
\end{split}
\end{equation}
On the other hand, we have
\begin{equation}\label{9-15-10}
\begin{split}
&g_m\bigl(
\frac{1}{|y-x_{j, L}|^{N-2m+2h}}, \frac{1}{|y-x_{j, L}|^{N-2m+2k}} \bigr)\\
=&\sum_{l=1}^{2m-1} \bar l_l( y-x_{j,L},  \nabla^l\frac{1}{|y-x_{j, L}|^{N-2m+2h}}, \nabla^{2m-l }\frac{1}{|y-x_{j, L}|^{N-2m+2k}})\\
&+\sum_{l=0}^{2m-1} \tilde  l_l(  \nabla^l\frac{1}{|y-x_{j, L}|^{N-2m+2h}}, \nabla^{2m-l-1 }\frac{1}{|y-x_{j, L}|^{N-2m+2k}})
\\
=&\frac{1}{|y-x_{j, L}|^{2N-2m+2(h+k)-1}}\bar f_{h,k}(\omega),\quad \omega\in \mathbb S^{N-1},
\end{split}
\end{equation}
where $\bar f_{h,k}$ is some function defined on $\mathbb S^{N-1}$.  Thus \eqref{8-15-10}
and \eqref{9-15-10} yield
\begin{equation}\label{10-15-10}
\begin{split}
\int_{\partial B_\delta(x_{j, L})}g_m\bigl(
\frac{1}{|y-x_{j, L}|^{N-2m+2h}}, \frac{1}{|y-x_{j, L}|^{N-2m+2k}} \bigr)
=\frac{1}{\theta^{N-2m+2(h+k)}}\int_{\mathbb S^{N-1}}
\bar f_{h,k}(\omega).
\end{split}
\end{equation}
Since the left hand side of \eqref{10-15-10} is finite and $N-2m+2(h+k)>0$, we conclude
$
\int_{\mathbb S^{N-1}}
\bar f_{h,k}(\omega)=0,
$
which gives
\begin{equation}\label{11-15-10}
 \int_{\partial B_\delta(x_{j, L})}g_m\bigl(
\frac{1}{|y-x_{j, L}|^{N-2m+2h}}, \frac{1}{|y-x_{j, L}|^{N-2m+2k}} \bigr)
=0.
\end{equation}
Therefore, we have proved
\begin{equation}\label{12-15-10}
 \int_{\partial B_\delta(x_{j, L})}g_m(V_{j, L}, V_{j, L})
=0.
\end{equation}
Inserting \eqref{6-15-10} and \eqref{12-15-10}
 into \eqref{1-15-10}, we obtain
\begin{equation}\label{13-15-10}
\begin{split}
& \int_{\partial B_\delta(x_{j, L})}g_m(Q_{j, L}, Q_{j, L})\\
 =&\sum_{i\ne j}\frac{B'_m}{\mu_{j, L}^{\frac{N-2m}2}\mu_{i, L}^{\frac{N-2m}2}|x_{j, L}-x_{i, L}|^{N-2m}
}+ O\bigl(\frac1{\mu_L^{N-2m+2} L^{N-2m}}+\frac1{\mu_L^{N-2m} L^{2(N-2m)}}\bigr).
\end{split}
\end{equation}
From \eqref{20-14-10}  and \eqref{13-15-10}, we get
\begin{equation}\label{left}
\begin{split}
\text{LHS of \eqref{2-4-9}}
=&\sum_{i\ne j}\frac{B'_m}{\mu_{j, L}^{\frac{N-2m}2}\mu_{i, L}^{\frac{N-2m}2}|x_{j, L}-x_{i, L}|^{N-2m}
}\\
&+O\Bigl( \frac1{\mu_L^{N}}+\frac1{\mu_L^{N-2m+2} L^{N-2m}}+\frac{\max_i|x_{i,L}-P_i|^\beta}{\mu_L^{\frac{N-2m}{2}+\tau}}\\
&+\frac{\max_i|x_{i,L}-P_i|^{2\beta}}{\mu_L^{2\tau}}+\frac1{\mu_L^{N-2m} L^{2(N-2m)}}\Bigr).
\end{split}
\end{equation}

We now
estimate the right hand side of
of  \eqref{2-4-9}. Firstly,
we have
\begin{equation}\label{1-10-9}
\int_{\partial B_\delta(x_{j, L})}K(y) u_{L}^{m^*} \bigl\langle y-x_{j, L}, \nu\bigr\rangle=O\bigl(
\frac1{\mu_L^N}\bigr).
\end{equation}
On the other hand, \,\,\,\,\,\,\,\,\,\,
\begin{equation}\label{nn10-10-9}
 \begin{split}
&-
\frac1{m^*}
\int_{ B_\delta(x_{j, L})} \bigl\langle \nabla K(y),  y-x_{j, L}\bigr\rangle u_{L}^{m^*}\\
=&
-\frac \beta{m^*}
\int_{ B_\delta(x_{j, L})} \sum\limits_{i=1}^N a_i|y_i-x_{j,i,L}|^{\beta} u_{L}^{m^*}\\
&+
O\bigl(
\frac{\max_i |x_{i, L}- P_i|^{}}{\mu_L^{\beta-1}}+\frac{\max_i |x_{i, L}- P_i|^{\beta-1}}{\mu_L}
+\frac{\max_i |x_{i, L}- P_i|^{m^*\beta}}{\mu_L^{m^* \tau}}+\frac1{\mu_{j, L}^{\beta+1}}\bigr)\\
=& -\frac {\beta \sum\limits_{i=1}^N a_i}{m^*N\mu_{j, L}^{\beta}}\int_{\mathbb R^N} |y|^\beta U_{0,1}^{m^*}\\
& +
O\Bigl(
\frac{\max_i |x_{i, L}- P_i|^{}}{\mu_L^{\beta-1}}+\frac{\max_i |x_{i, L}- P_i|^{\beta-1}}{\mu_L}
+\frac{\max_i |x_{i, L}- P_i|^{m^*\beta}}{\mu_L^{m^* \tau}}+\frac1{\mu_{j, L}^{\beta+1}}\Bigr).
\end{split}
\end{equation}

 Thus, the desired result follows from \eqref{left}, \eqref{1-10-9} and \eqref{nn10-10-9}.
\end{proof}

\begin{proposition}\label{p1-18-9}
 It holds  $
|x_{j, L}- P_j|=O\bigl(\frac1{\mu_{j,L}^2} \bigr).$
\end{proposition}

\begin{proof}
With the same arguments as those of \cite{DLY}, we can  verify the estimates  $\mu_{L}|x_{j, L}- P_j|\le C$ and $\frac{\mu_L^\beta}{\mu_L^{N-2m}L^{N-2m}}\le C$ for some $0<C<\infty$. Hence we may assume
$\mu_{L}(x_{j, L}- P_j)\to x_0$, which implies that
\begin{equation}\label{7-8-9}
 \begin{split}
&\int_{ B_\delta(x_{j, L})}\frac{\partial K(y)}{\partial y_i}
U_{x_{j,L}, \mu_{j,L}}^{m^*} \\
=&a_i \beta\int_{ B_{\delta\mu_{j,L}}(0)} |\mu_{j,L}^{-1} y_i+x_{j,i,L}-P_{j,i}
|^{\beta-2} (\mu_{j,L}^{-1} y_i+x_{j,i, L}-P_{j, i})U_{0,1}^{m^*} +
O\bigl(\frac1{\mu_{j,L}^\beta} \bigr)\\
=&\frac {a_i\beta}{\mu_{j, L}^{\beta-1}}\Bigl(\int_{ B_{\delta\mu_{j,L}}(0)}
|y_i+x_{0,i}|^{\beta-2}(y_i+x_{0, i})U_{0,1}^{m^*}+o(1)\Bigr) +
O\bigl(\frac1{\mu_{j,L}^\beta} \bigr).
\end{split}
\end{equation}

 It follows from  \eqref{4-8-9} and \eqref{7-8-9} that
\begin{equation}\label{8-8-9}
\int_{ B_{\delta\mu_{j,L}}(0)} |y_i+x_{0,i}|^{\beta-2}(y_i+x_{0,
i})U_{0,1}^{m^*}=o(1),
\end{equation}
which yields   $x_0=0$.

 Noting that
$
\int_{ B_{\delta\mu_{j,L}}(0)} | y_i
|^{\beta-2}  y_iU_{0,1}^{m^*}=0,
$
 we obtain
\begin{equation}\label{9-8-9}
 \begin{split}
&\int_{ B_\delta(x_{j, L})}\frac{\partial K(y)}{\partial y_i}
U_{x_{j,L}, \mu_{j,L}}^{m^*} \\
=& \frac{a_i\beta(\beta-1)}{\mu_{j, L}^{\beta-1}}\int_{ B_{\delta\mu_{j,L}}(0)} |y_i
|^{\beta-2} U_{0,1}^{m^*} \mu_{j, L}(x_{j,i, L}-P_{j, i})+
O\bigl(\frac1{\mu_{j,L}^\beta} \bigr),
\end{split}
\end{equation}
which, together  with Lemma~\ref{l1-26-9},  implies
$
\mu_{j, L}(x_{j,i, L}-P_{j, i})=O\bigl(\frac1{\mu_{j,L}} \bigr).
$
Therefore
$
|x_{j, L}- P_j|=O\bigl(\frac1{\mu_{j,L}^2} \bigr).
$
\end{proof}

\begin{proposition}\label{p2-18-9}
 It holds
$
\mu_{j, L}= L^{\frac{N-2m}{\beta-N+2m}} \Bigl( \bar B_j+ O\bigl(\frac1{L^{\frac{N-2m}{\beta-N+2m}}  }\bigr)\Bigr),
$
for some constant $\bar B_j>0$.
\end{proposition}

\begin{proof}
Noting that $|x_{j, L}-P_j|=O\bigl(\frac1{\mu_{j,L}^2} \bigr)$,
 we find
\[
\begin{split}
\frac1{|x_{i, L}-x_{j, L}|^{N-2m}}=\frac1{(|P_{i}-P_{j}|+O\bigl(\frac1{\mu_L^2}\bigr) )^{N-2m}}=
\frac1{|\tilde P_i-\tilde P_j|^{N-2m} L^{N-2m}}
\bigl( 1+ O\bigl(
\frac1{|\tilde P_i-\tilde P_j|L\mu_L^2}\bigr)\bigr).
\end{split}
\]
As a result, we see that \eqref{10-10-9} is equivalent to
\begin{equation}\label{11-10-9}
 \frac {1}{\mu_{j, L}^{\beta}}=B\sum_{i\ne j} \frac1{\mu_{i, L}^{\frac{N-2m}2}\mu_{j, L}^{\frac{N-2m}2}
L^{N-2m}|\tilde P_i-\tilde P_j|^{N-2m}}+
O\bigl(\frac{1}{\mu_L \mu_L^{N-2m}L^{N-2m}}+\frac1{\mu_{j, L}^{\beta+1}}\bigr).
\end{equation}
Since we assume that $0<\mu_0\le \frac{\mu_{i, L}}{\mu_{j, L}}\le \mu'_0$, we can easily deduce from
\eqref{11-10-9} that
\[
 0<c'_0 L^{\frac{N-2m}{\beta-N+2m}}\le \mu_{i, L}\le c'_1 L^{\frac{N-2m}{\beta-N+2m}}.
 \]
Let $\frac1{\mu_{i, L}^{\frac{N-2m}2}}= \frac{a_{i, L}}{ L^{\frac{(N-2m)^2}{2(\beta-N+2m)}}  } $. Then,
$0<c_0\le a_{j, L}\le c_1<+\infty$, and
\[
a_{j, L}^\kappa = B\sum_{i\ne j} \frac{a_{i, L}}{|\tilde P_i-\tilde P_j|^{N-2m}}+
O\bigl(\frac1{\mu_{j, L}}\bigr),
\]
where $\kappa=\frac{\beta-\frac{N-2m}2}{\frac{N-2m}2}>1.  $ Hence, from Lemma~\ref{100-l1-13-9}, we obtain
$
a_{j, L}=a_j+
O\bigl(\frac1{\mu_{j, L}}\bigr)
$ for some $a_j>0$.
\end{proof}

\subsection{Local uniqueness }
Suppose that  problem $(P)$ have two different solutions $u^{(1)}_{L}  $  and
$u^{(2)}_{L} $, which blow up at $P_j$, $j=1,  2, \cdots $.  For $k=1, 2$, we  use
$x_{j, L}^{(k)}$ and $\mu_{j,L}^{(k)}$ to denote the center and the height of the bubbles appearing
in $u^{(k)}_{L},$ respectively.

  Let
\begin{equation}\label{1-7-4}
\eta_L =\frac{u^{(1)}_{L}-u^{(2)}_{L}}{\|u^{(1)}_{L}-u^{(2)}_{L}\|_{*}}.
\end{equation}
Then, $\eta_L$ satisfies $\|\eta_L\|_{*}=1$ and
\begin{equation}\label{2-7-4}
(-\Delta)^m \eta_L = f(y, u^{(1)}_{L}, u^{(2)}_{L}),
\end{equation}
where
\begin{equation}\label{3-7-4}
f(y, u^{(1)}_{L}, u^{(2)}_{L})=\frac1{\|u^{(1)}_{L}-u^{(2)}_{L}\|_{*}}
K(y) \bigl( (  u^{(1)}_{L})^{m^*-1}- (  u^{(2)}_{L})^{m^*-1}  \bigr).
\end{equation}
Write
\begin{equation}\label{n3-7-4}
f(y, u^{(1)}_{L}, u^{(2)}_{L})= K(y)c_{L}(y) \eta_L(y),
\end{equation}
where
\begin{equation}\label{nn3-7-4}
 c_{L}(y)=(m^*-1) \int_0^1 \bigl( t  u^{(1)}_{L}(y) +(1-t)  u^{(2)}_{L}(y)
 \bigr)^{m^*-2}\,dt.
\end{equation}
It follows from Propositions~\ref{p1-18-9} and \ref{p2-18-9} that
\[
 \begin{split}
 U_{ x^{(1)}_{i, L}, \mu^{(1)}_{i, L}}-  U_{ x^{(2)}_{i, L}, \mu^{(2)}_{i, L}}=& O\Bigl( |x^{(1)}_{i, L}
 -x^{(2)}_{i, L}||\nabla U_{ x^{(1)}_{i, L}, \mu^{(1)}_{i, L}}|+ |\mu^{(1)}_{i, L}
 -\mu^{(2)}_{i, L}||\partial_\mu
 U_{ x^{(1)}_{i, L}, \mu^{(1)}_{i, L}}|\Bigr)\\
 =&
 O\Bigl( \frac1{\mu_{ L}}
 U_{ x^{(1)}_{i, L}, \mu^{(1)}_{i, L}}\Bigr),
 \end{split}
 \]
which gives
\begin{equation}\label{4-18-9}
u_L^{(1)} - u_L^{(2)}= O\Bigl( \frac1{\mu_L} \sum_{i=1}^
\infty U_{ x^{(1)}_{i, L}, \mu^{(1)}_{i, L}}+|\omega_L^{(1)}|+|\omega_L^{(2)}|
 \Bigr).
\end{equation}
Thus, we have proved
\begin{equation}\label{3-18-9}
\begin{split}
 c_{L}(y)=&(m^*-1) U^{m^*-2}_{ x^{(1)}_{j, L}, \mu^{(1)}_{j, L}}\\
 &+O\Bigl( \bigl( \frac1{\mu_L} U_{ x^{(1)}_{j, L}, \mu^{(1)}_{j, L}}+\frac1{\mu_L^{\frac{N-2m}2}L^{N-2m}}+|\omega_L^{(1)}|+|\omega_L^{(2)}|\bigr)^{m^*-2}
 \Bigr),
 \quad y\in B_d(x^{(1)}_{j, L}).
 \end{split}
\end{equation}

 Using the H\"older inequality, noting that $\|\omega^{(i )}_L\|_*=o(1)$, we can deduce
\begin{equation}\label{nn1-11-9}
\begin{split}
&|f(z, u^{(1)}_{L}, u^{(2)}_{L})|\le  C W_{L,{\bf x^{(1)},\mu^{(1)}}}^{m^*-2} |\eta_L(z)|+C (|\omega^{(1)}_L(z)|^{m^*-2}
+|\omega^{(2)}_L(z)|^{m^*-2})|\eta_L(z)|\\
\le &C W_{L,{\bf x^{(1)},\mu^{(1)}}}^{m^*-2} |\eta_L(z)|+ o(1) \|\eta_L\|_*\sigma(z)\Big(\sum\limits_{j=1}^\infty\frac{(\mu^{(1)}_{j,L})^{\frac{N-2m}2}}{(1+\mu^{(1)}_{j,L}
|y-x^{(1)}_{j,L}|)^{\frac{N-2m}{2}+\tau}}\Big)^{m^*-1}\\
\le & C W_{L,{\bf x^{(1)},\mu^{(1)}}}^{m^*-2} |\eta_L(z)|+ o(1) \|\eta_L\|_*\sigma(z)\sum\limits_{j=1}^\infty\frac{(\mu^{(1)}_{j,L})^{\frac{N+2m}2}}{(1+\mu^{(1)}_{j,L}
|y-x^{(1)}_{j,L}|)^{\frac{N+2m}{2}+\tau}},
\end{split}
\end{equation}
where $W_{L,{\bf x,\mu}}$ is defined by \eqref{10-22-80} in Appendix A.

 Similar to the proof of Lemma~\ref{lem:4},  we deduce from \eqref{nn1-11-9} that
\begin{equation}\label{1-11-9}
\begin{aligned}
&\quad |\eta_L(y)|\Bigl(\sigma(y)\sum\limits_{i=1}^\infty
\frac{(\mu^{(1)}_{i,L})^{\frac{N-2m}2}}{(1+\mu^{(1)}_{i,L}|y-x^{(1)}_{i,L}|)^{\frac{N-2m}{2}+\tau}}\Bigr)^{-1}\\
&\leq
C\Bigl(o(1)||\eta_L||_{*}
+ \frac{ \sum\limits_{j=1}^\infty\frac{1}{(1+\mu^{(1)}_{i,L}|z-x^{(1)}_{j,L}|)
^{\frac{N-2m}{2}+\tau+\tilde \theta}}  }
{ \sum\limits_{j=1}^\infty\frac{1}{(1+\mu^{(1)}_{i,L}|z-x^{(1)}_{j,L}|)
^{\frac{N-2m}{2}+\tau}} }   ||\eta_L||_{*}\Bigr).
\end{aligned}
\end{equation}

To obtain a contradiction, we just need to show that $|\eta_L(y)|=o(1)$ in $
\cup_{j} B_{R (\mu^{(1)}_{j, L})^{-1}}
(x^{(1)}_{j, L})$, which will be achieved by using the Pohozaev identities  in the small ball
$B_d
(x^{(1)}_{j, L}).$

 Let
\begin{equation}\label{1-10-4}
\tilde \eta_{L,j} (y)= \bigl(\frac1{\mu^{(1)}_{j, L}}\bigr)^{\frac{N-2m}2}\eta_L
(\frac{1}{\mu^{(1)}_{j, L}}y+ x^{(1)}_{j, L}).
\end{equation}

\begin{lemma}\label{l1-10-4}
It holds
\begin{equation}\label{10-9-4}
\tilde \eta_{L,j}(y) \to \sum_{k=0}^Nb_{j, k}\psi_k(y), \hbox{ as } L\to \infty,
\end{equation}
uniformly in $C^m(B_R(0))$ for any $R>0$, where $b_{j, k}$, $k=0, \cdots, N$,
are some constants,  and
\begin{equation}\label{5-8-4}
\psi_0 =\frac{\partial U_{0, \lambda}}{\partial \lambda}\Bigr|_{\lambda=1},\quad
\psi_j =\frac{\partial U_{0, 1}}{\partial y_j},\;\; j=1, \cdots,N.
\end{equation}

\end{lemma}

\begin{proof}
 In view of
$|\tilde \eta_{L,j}|\le C$ in any compact subset of $\mathbb R^N$, we may assume that $\tilde \eta_{L,j}\to \xi_j$ in
$ C_{loc}(\mathbb R^N)$.   Then it follows from the elliptic regularity theory and \eqref{2-7-4}  and \eqref{3-18-9} that
 $\xi_j$ satisfies
\begin{equation}\label{3-8-4}
(-\Delta)^m \xi_j = (m^*-1) U_{0, 1}^{m^*-2}\xi_j, \quad \text{in}\; \mathbb R^N,
\end{equation}
which combining with the non-degeneracy of $U_{0,1}$ gives
\begin{equation}\label{4-8-4}
\xi_j =\sum_{k=0}^Nb_{j, k}\psi_k.
\end{equation}
\end{proof}

 Let $G(y, x)=C_m|y-x|^{2m-N}$ be the corresponding Green's  function of $(-\Delta)^m$ in $\R^N$.

\begin{lemma}\label{l1-9-4}

We have the following estimate:
\begin{equation}\label{1-9-4}
\begin{split}
  \eta_L(x) = &  \sum_{j=1}^\infty\sum_{|\alpha|=0}^{2m-1} A_{j, L, \alpha} \partial^\alpha G(x^{(1)}_{j,  L}, x)
+O\bigl( \frac{1 }{\mu_L^{\frac{N+2m}2-\theta}} \bigr)\\
:=& \sum_{j=1}^\infty F_{j, m, L}(x)+O\bigl( \frac{1 }{\mu_L^{\frac{N+2m}2-\theta}} \bigr),\quad
\text{\hbox{in} $C^{2m-1}\bigl(\mathbb R^N \setminus \cup_{j=1}^\infty  B_{2\sigma}(x^{(1)}_{j, L})\bigr)$},
\end{split}
\end{equation}
where $\sigma>0$ is any   small constant, $\partial^\alpha  G(y, x)=\frac{\partial^{|\alpha|} G(y,x)}{\partial y_1^{\alpha_{i, 1}}\cdots\partial y_N^{\alpha_{i, N}}}$,
$\alpha=(\alpha_{i,1},\cdots, \alpha_{i,N}   )$, and the constants $A_{j, L, \alpha}$
satisfy the following estimates:
\begin{equation}\label{Ah}
A_{j, L, 0}= \int_{B_\sigma(x^{(1)}_{j, L})}   f(y, u^{(1)}_{L}(y), u^{(2)}_{L}(y))\,dy=o\bigl(\frac1{\mu_L^{\frac{N-2m}2}}\bigr),
\end{equation}
\begin{equation}\label{Eh}
A_{j, L, \alpha}=O\bigl(\frac1{\mu_L^{\frac{N-2m}2+|\alpha|}}\bigr),\quad |\alpha|\ge 1.
\end{equation}
\end{lemma}
\begin{proof}
Denote $ f_L^*(y)=   f(y, u^{(1)}_{L}(y), u^{(2)}_{L}(y))$. We have
\begin{equation}\label{02-10-4}
\begin{split}
\eta_L(x)
= &\int_{\mathbb R^N} G(y,x) f_L^*(y)\,dy\\
=& \sum_{j=1}^\infty \int_{ B_\sigma(x^{(1)}_{j, L})} G(y,x) f_L^*(y)\,dy+
\int_{\mathbb R^N\setminus \cup_j B_\sigma(x^{(1)}_{j, L})} G(y,x) f_L^*(y)\,dy\\
=& \sum_{j=1}^\infty \sum_{|\alpha|=0}^{2m-1} A_{j, L, \alpha} \partial^\alpha G(x^{(1)}_{j,  L}, x)
 \\
 &+
 \sum_{j=1}^\infty O\Bigl(\int_{B_\sigma(x^{(1)}_{j, L})}|y-x^{(1)}_{j, L}|^{2m}| f^*_L(y)|\,dy\Bigr)+ \int_{\mathbb R^N\setminus \cup_j B_\sigma(x^{(1)}_{j, L})} G(y,x) f_L^*(y)\,dy.
 \end{split}
\end{equation}

 For $y\in \mathbb R^N\setminus \cup_j B_\sigma(x^{(1)}_{j, L})$,
noting that $\tau=\frac{N-2m}2-\vartheta$ for $\vartheta>0$ small, similarly to \eqref{nn1-11-9}
and \eqref{10-22-1},
 we
 find
\[
 \begin{split}
 |f_L^*(y)|
 \le &  C \mu_L^{-\tau}\bigl( \frac1{\mu_L^{2m}}+(\mu_L^{-\tau}\|\omega^{(1)}_L\|_*)^{m^*-2}\bigr) \Bigl(
 \sum_{j=1}^\infty \frac1{|y- x^{(1)}_{j, L}|^{\frac{N-2m}2+\tau}}\Bigr)^{m^*-1}\\
 \le &C \mu_L^{-\frac{N+2m}2+\vartheta}
 \sum_{j=1}^\infty \frac1{|y- x^{(1)}_{j, L}|^{\frac{N+2m}2+\tau}}.
 \end{split}
 \]
Thus, we  have
\begin{equation}\label{1-17-9}
 \int_{\mathbb R^N\setminus \cup_j B_\sigma(x^{(1)}_{j, L})} G(y,x) f_L^*(y)\,dy\le
 C \mu_L^{-\frac{N+2m}2+\vartheta}
 \sum_{j=1}^\infty \frac1{|x- x^{(1)}_{j, L}|^{\frac{N-2m}2+\tau}}\le
 \frac{C }{\mu_L^{\frac{N+2m}2-\vartheta}}.
\end{equation}
Similarly, by Lemma~\ref{lem:3}
\begin{equation}\label{(IV)}
\begin{split}
&   \int_{B_\sigma(x^{(1)}_{j, L})}|y-x_{j, L}^{(1)}|^{2m} |f^*_L|\\
\le & C\int_{B_\sigma(x^{(1)}_{j, L})}|y-x_{j, L}^{(1)}|^{2m}
\frac{(\mu^{(1)}_{j,L})^{\frac{N+2m}2}}{ (1+\mu^{(1)}_{j,L}|y- x^{(1)}_{j, L}|)^{(\frac{N-2m}2+\tau)(m^*-1)}}
\\
\le &  \frac C{(\mu^{(1)}_{j,L})^{\frac{N+2m}2-(m^*-1)\vartheta}}
\int_{B_\sigma(x^{(1)}_{j, L})}
\frac{1}{ |y- x^{(1)}_{j, L}|^{N-\vartheta(m^*-1)}}
\le \frac C{\mu_L^{\frac{N+2m}2-(m^*-1)\vartheta}}.
\end{split}
\end{equation}

Inserting \eqref{(IV)} and \eqref{1-17-9} into  \eqref{02-10-4}, we obtain \eqref{1-9-4}.
Similarly, we can prove that \eqref{1-9-4} holds in $C^{2m-1}\bigl(\mathbb R^N \setminus \cup_{j=1}^\infty  B_{2\sigma}(x^{(1)}_{j, L})\bigr)$.
It remains to estimate $A_{j, L, \alpha}$.
\begin{equation}\label{A}
\begin{split}
A_{j, L, 0}=&\int_{B_\sigma(x^{(1)}_{j, L})}   f(y, u^{(1)}_{L}(y), u^{(2)}_{L}(y))\,dy\\
=& \frac1{(\mu_{j, L}^{(1)})^{N-2m}}\int_{B_R(0)} \frac1{(\mu_{j, L}^{(1)})^{2m}}
f^*_L( \frac1{\mu^{(1)}_{j, L}} y+ x^{(1)}_{j, L} )  \,dy\\
&+\frac1{(\mu_{j, L}^{(1)})^{\frac{N-2m}2}}
O\Bigl( \int_{ B_{\sigma \mu_{j,L}^{(1)}}(0)\setminus B_R(0) }
 \frac1{|y|^{(N-2m-\vartheta)(m^*-1)}}\,dy \Bigr)
\\
=& \frac1{(\mu_{j, L}^{(1)})^{\frac{N-2m}2}}\Bigl( m^*-1)\int_{\mathbb R^N} U_{0,1}^{2^*(m)-2}
\sum_{k=0}^N b_{j, k} \psi_k  +o(1)\Bigr)=o\bigl(\frac1{\mu_L^{\frac{N-2m}2}}\bigr).
\end{split}
\end{equation}
If $|\alpha|\ge 1$, then
\begin{equation}\label{10-16-10}
|A_{j, L, \alpha}|\le C\int_{B_\sigma(x^{(1)}_{j, L})} |y-x^{(1)}_{j, L}|^{|\alpha|}|   f(y, u^{(1)}_{L}(y), u^{(2)}_{L}(y))|=O\bigl(\frac1{\mu_L^{\frac{N-2m}2+|\alpha|}}\bigr).
\end{equation}
\end{proof}

 Using \eqref{1-4-9} and \eqref{2-4-9}, we can deduce the following
  identities:
\begin{equation}\label{14-11-9}
 \begin{split}
&\int_{\partial B_{d}(x^{(1)}_{j, L})}f_{m, i}(\eta_L, u^{(1)}_L)
+\int_{\partial B_{d}(x^{(1)}_{j, L})}f_{m, i}( u^{(2)}_L, \eta_L)\\
=&
\int_{\partial B_d(x^{(1)}_{j, L})}K(y) C_L(y)  \eta_{L}\nu_i-
\int_{ B_d(x^{(1)}_{j, L})}\frac{\partial K(y)}{\partial y_i} C_L(y) \eta_{L},
\end{split}
\end{equation}
and
\begin{equation}\label{15-11-9}
 \begin{split}
&\int_{\partial B_{d}(x^{(1)}_{j, L})}g_{m}(\eta_L, u^{(1)}_L)
+\int_{\partial B_{d}(x^{(1)}_{j, L})}g_{m}( u^{(2)}_L, \eta_L)
\\
=&
\int_{\partial B_d(x^{(1)}_{j, L})}K(x) C_L(y) \eta_{L} \bigl\langle y-x^{(1)}_{j, L}, \nu\bigr\rangle-
\int_{ B_d(x^{(1)}_{j, L})} \bigl\langle \nabla K(y),  y-x^{(1)}_{j, L}\bigr\rangle C_L(y) \eta_{L},
\end{split}
\end{equation}
where $C_L(y) =\int_0^1 ( t  u^{(1)}_{L}+ (1-t)  u^{(2)}_{L})^{m^*-1}\,dt$  and $d>0$ is a   small constant.

 Similar to
\eqref{3-18-9}, we can deduce
\begin{equation}\label{5-18-9}
 C_{L}(y)= U^{m^*-1}_{ x^{(1)}_{j, L}, \mu^{(1)}_{j, L}}
 +O\Bigl( \bigl( \frac1{\mu_L} U_{ x^{(1)}_{j, L}, \mu^{(1)}_{j, L}}+\frac1{\mu_L^{\frac{N-2m}2}L^{N-2m}}+|\omega_L^{(1)}|+|\omega_L^{(2)}|\bigr)^{m^*-1}
 \Bigr),
 \;\; y\in B_d(x^{(1)}_{j, L}).
\end{equation}
To estimate the boundary terms in \eqref{14-11-9} and \eqref{15-11-9}, we need the following estimates
which can be deduced from Proposition~\ref{p1-6-3}  and Lemma~\ref{l1-9-4}.
\begin{equation}\label{61-4}
\begin{split}
u_L^{(k)}(y)=&\frac{\tilde C_m}{(\mu_{j, L}^{(k)})^{\frac{N-2m}2}} \sum_{i=0}^{m-1}
\frac{\alpha_i}{(\mu_{j, L}^{(k)})^{2i}|y-x^{(k)}_{j, L}|^{N-2m+2i}}+O\bigl(\frac1{\mu_L^{\frac{N-2m}2} L^{N-2m}}+\frac1{\mu_L^{\frac{N+2m}2}}\bigr)\\
=:&V_L^{(k)}+O\bigl(\frac1{\mu_L^{\frac{N-2m}2} L^{N-2m}}+\frac1{\mu_L^{\frac{N+2m}2}}\bigr)
, \quad y\in \partial B_d(x^{(k)}_{j, L}),
\end{split}
\end{equation}
and
\begin{equation}\label{eta}
  \eta_L(y) = F_{j, m, L}(y)
+O\bigl( \frac{1 }{\mu_L^{\frac{N+2m}2-\theta}} \bigr)+o\bigl(\frac{1}{\mu^{\frac{N-2m}2}L^{N-2m}} \bigr), \quad y\in \partial B_d(x^{(1)}_{j, L}).
\end{equation}

\begin{proof}[Proof of Theorem~\ref{main}    ]

 {\bf Step~1.}  We prove $b_{j, k}=0$, $k=1, \cdots, N$. We need to estimate each term in \eqref{14-11-9}. From
\eqref{5-18-9}, we obtain
\begin{equation}\label{16-11-9}
\int_{\partial B_d(x^{(1)}_{j, L})}K(y) C_L(y)  \eta_{L}\nu_i=O\bigl(
\frac1{\mu_L^{N}}\bigr),
\end{equation}
and
\begin{equation}\label{17-11-9}
 \begin{split}
&-
\int_{ B_d(x^{(1)}_{j, L})}\frac{\partial K(y)}{\partial y_i} C_L(y) \eta_{L}\\
=&-\frac{\beta a_i}{(\mu_{j,L}^{(1)})^{\beta-1+\frac{N+2m}2}}
\int_{ B_{d \mu_{j,L}^{(1)}  }(0)}|y_i|^{\beta-2} y_i C_L(\frac1{\mu^{(1)}_{j, L}}y+ x^{(1)}_{j, L})\tilde \eta_{L,j}+O\bigl(
\frac1{\mu_{ L}^{\beta}}\bigr)\\
=&-\frac{\beta a_i}{(\mu^{(1)}_{j, L})^{\beta-1}}\Bigl(
\int_{ B_R(0)}|y_i|^{\beta-2} y_i U_{0, 1}^{m^*-1}\sum_{k=0}^N b_{j, k} \psi_k+ o(1)\Bigr)+O\bigl(
\frac1{\mu_L^{\beta}}\bigr)\\
=&-\frac{\beta a_i }{(\mu^{(1)}_{j, L})^{\beta-1}}\Bigl(
b_{j, i}\int_{ \mathbb R^N}| y_i|^{\beta-2} y_i U_{0,1}^{m^*-1} \psi_i+ o(1)\Bigr)+O\bigl(
\frac1{\mu_L^{\beta}}\bigr).
\end{split}
\end{equation}

 Combining \eqref{16-11-9}  and \eqref{17-11-9}, we are led to
\begin{equation}\label{10-14-9}
 \begin{split}
\text{RHS of \eqref{14-11-9}}
= -\frac{\beta a_i} {(\mu^{(1)}_{j, L})^{\beta-1}}\Bigl(
b_{j, i}\int_{ \mathbb R^N}|y_i|^{\beta-2} y_i U_{0,1}^{2^*(m)-1} \psi_i+ o(1)\Bigr)+O\bigl(
\frac1{\mu_L^{\beta}}\bigr).
\end{split}
\end{equation}
To estimate the left hand side of \eqref{14-11-9}, using \eqref{61-4},  we have

\begin{equation}\label{30-15-10}
 \begin{split}
 \text{LHS of \eqref{14-11-9}}
=&\int_{\partial B_d(x^{(1)}_{j, L})}f_{m, i}(F_{j, m, L}, V^{(1)}_L)
+\int_{\partial B_d(x^{(1)}_{j, L})}f_{m, i}( V^{(2)}_L, F_{j, m, L})\\
&+o\bigl(\frac1{\mu_L^{N-2m} L^{N-2m}}\bigr)+O\bigl(\frac1{\mu_L^{N}}\bigr).
\end{split}
\end{equation}

Now we claim that

\begin{equation}\label{60-16-10}
 \int_{\partial B_d(x^{(1)}_{j, L})}f_{m, i}(F_{j, m, L}, V^{(1)}_L)=0,\quad
\int_{\partial B_d(x^{(1)}_{j, L})}f_{m, i}( V^{(2)}_L, F_{j, m, L})=0.
\end{equation}
This  gives

\begin{equation}\label{61-16-10}
\text{LHS of \eqref{14-11-9}}=o\bigl(\frac1{\mu_L^{N-2m} L^{N-2m}}\bigr)+O\bigl(\frac1{\mu_L^{N}}\bigr).
\end{equation}
Hence, \eqref{10-14-9} and \eqref{61-16-10} imply $b_{j, i}=0$, $i=1, \cdots, N$.

 It remains to prove \eqref{60-16-10}. We just prove the second integral in \eqref{60-16-10}
is zero. Note that this integral is a linear combination of the following integrals
\begin{equation}\label{nn60-16-10}
\int_{\partial B_d(x^{(1)}_{j, L})}f_{m, i}\bigl(\partial^\alpha G(x^{(2)}_{j,  L} ,x),
\frac{1}{|y-x^{(2)}_{j, L}|^{N-2m+2i}} \bigr).
\end{equation}
So, we just need to prove that the integral defined in \eqref{nn60-16-10} is zero.

  Since  $\frac{1}{|y-x^{(2)}_{j, L}|^{N-2m+2i}}$ and $\partial^\alpha G(x^{(2)}_{j,  L}, x)$ are $m$-harmonic in $
B_d(x^{(1)}_{j, L})\setminus B_\theta(x^{(2)}_{j, L})$, by Proposition~\ref{p1-13-10}, we find
\begin{equation}\label{62-16-10}
\begin{split}
&\int_{\partial B_d(x^{(1)}_{j, L})}f_{m, i}\bigl(\partial^\alpha G(x^{(2)}_{j,  L} ,x),
\frac{1}{|y-x^{(2)}_{j, L}|^{N-2m+2i}} \bigr)\\
=&
\int_{\partial B_\theta(x^{(2)}_{j, L})}f_{m, i}\bigl(\partial^\alpha G(x^{(2)}_{j,  L} ,x),
\frac{1}{|y-x^{(2)}_{j, L}|^{N-2m+2i}} \bigr).
\end{split}
\end{equation}
By using the same arguments as in \eqref{8-15-10}--\eqref{11-15-10}, we can prove
\begin{equation}\label{62-16-10}
\int_{\partial B_\theta(x^{(2)}_{j, L})}f_{m, i}\bigl(\partial^\alpha G(x^{(2)}_{j,  L} ,x),
\frac{1}{|y-x^{(2)}_{j, L}|^{N-2m+2i}} \bigr)=0.
\end{equation}

{\bf Step~2}. We prove $b_{j, 0}=0$. It is easy to deduce
\begin{equation}\label{1-12-9}
 \begin{split}
&\text{RHS of \eqref{15-11-9}}\\
= &-\frac{\beta}{(\mu_{j,L}^{(1)})^{\beta}}\Bigl(
\int_{ \mathbb R^N}\sum\limits_{i=1}^N a_i|y_i|^\beta U_{0, 1}^{m^*-1} \sum_{k=0}^N b_{j, k} \psi_k+o(1)\Bigr)+O\bigl(
\frac1{\mu_L^{N}}\bigr)\\
=&-\frac{\beta\sum\limits_{i=1}^N a_i}{N(\mu_{j,L}^{(1)})^{\beta}}\Bigl(b_{j, 0}
\int_{ \mathbb R^N}|y|^\beta U_{0,1}^{m^*-1}  \psi_0+o(1)\Bigr)+O\bigl(
\frac1{\mu_L^{N}}\bigr).
\end{split}
\end{equation}

It follows from  \eqref{61-4} and \eqref{eta} that

\begin{equation}\label{19-11-9}
\begin{split}
 \text{LHS of \eqref{15-11-9}}
=&\int_{\partial B_d(x^{(1)}_{j, L})}g_{m}(F_{j, m, L}, V^{(1)}_L)
+\int_{\partial B_d(x^{(1)}_{j, L})}g_{m}( V^{(2)}_L, F_{j, m, L})\\
&+o\bigl(\frac1{\mu_L^{N-2m} L^{N-2m}}\bigr)+O\bigl(\frac1{\mu_L^{N}}\bigr).
\end{split}
\end{equation}
 Similar to the proof of \eqref{60-16-10}   in {\bf Step~1}, we can prove
\begin{equation}\label{66-16-10}
 \int_{\partial B_d(x^{(1)}_{j, L})}g_{m}(F_{j, m, L}, V^{(1)}_L)=0,\quad
\int_{\partial B_d(x^{(1)}_{j, L})}g_{m}( V^{(2)}_L, F_{j, m, L})=0.
\end{equation}
Therefore
\begin{equation}\label{67-16-10}
\text{LHS of \eqref{15-11-9}}
=o\bigl(\frac1{\mu_L^{N-2m} L^{N-2m}}\bigr)+O\bigl(\frac1{\mu_L^{N}}\bigr).
\end{equation}
Combining \eqref{1-12-9}  and \eqref{67-16-10},  we are led to
\begin{equation}\label{4-12-9}
 b_{j, 0}
\int_{ \mathbb R^N}|y|^\beta U_{0, 1}^{m^*-1}  \psi_0=o(1).
\end{equation}
This gives $b_{j, 0}=0$.
\end{proof}

\begin{proof}[Proof of Theorem~\ref{th1-7-1}]
 To prove that $u_L$ is periodic in $y_1$,  we let
\[
 {v_L(y)}= u_L(y_1-L,y_2,\cdots,y_N).
 \]
  Then, $v_L$ is
 a bubbling solution whose blow-up set is the same as that of $u_L$. By the local uniqueness,
 $v_L=u_L$. Similarly, we can prove that $u_L$ is periodic in $y_j$, $j=2, \cdots, k$.
\end{proof}

\appendix

\numberwithin{equation}{section}

\section{Some basic estimates}
 In this section, we give some technical lemmas. Throughout  Appendixes  A, B, C and D, we will  use the same notations as before and we also use the same  $C$
to denote  different constants unless otherwise stated. The  proof of the following two
 Lemmas can be found in \cite{WY}.

\begin{lemma}\label{lem:1}
Let $x_i, x_j\in \mathbb{R}^N, x_i\not =x_j$, $i\ne j$, it holds
\[
\begin{split}
&\frac{1}{(1+|y-x_{i}|)^{\alpha}(1+|y-x_{j}|)^{\beta}}
\\
&
\leq\frac{C}{(1+|x_{i}-x_{j}|)^{\sigma}}
(\frac{1}{(1+|y-x_{j}|)^{\alpha+\beta-\sigma}}+\frac{1}{(1+|y-x_{i}|)^{\alpha+\beta-\sigma}}),
\end{split}
\]
where $\alpha$ and $\beta$ are some positive constants, $0<\sigma\le \min (\alpha, \beta)$.

\end{lemma}

\vskip8pt

\begin{lemma}\label{lem:2}
For any constant $0<\sigma<N-2m$, there exists  a constant $C=C(N, \sigma)>1$ such
that
$$
 \int_{\mathbb{R}^{N}}\frac{dz}{|y-z|^{N-2m}(1+|z|)^{2m+\sigma}}\leq\frac{C}{(1+|y|)^{\sigma}}.$$

\end{lemma}

Set
$$\Omega_i:=\{y\in \mathbb{R}^N\hbox{ such that } |y-x^i|\leq |y-x^j|, \hbox{ for all }j\not =i\}.$$

\begin{lemma}\label{lem:3}
For any $\theta>k,$ there exists a constant $C$,  such that
\begin{equation}\label{e1}
 \sum_j\frac{1}{(1+\mu_j|y-x_j|)^{\theta}}\leq\frac{C}{(1+\mu_i|y-x_i|)^{\theta}},\quad y\in B_i:=B_1(x_i).
\end{equation}

\end{lemma}

\begin{proof}
 For $y\in B_1(x_i)$, it holds $|y-x_j|\ge \frac12 |x_i -x_j|$. As a result,
\[
  \sum_{j\ne i} \frac{1}{(1+\mu_j|y-x_j|)^{\theta}}\le \sum_{j\ne i} \frac{C}{(\mu_j|x_i-x_j|)^{\theta}}\le
  \frac{C}{(\mu_i L)^{\theta}}\le \frac{C}{(1+\mu_i|y-x_i|)^{\theta}}.
  \]
\end{proof}

 The following lemma is the main ingredient in the discussion of the existence
 and the local uniqueness of  bubbling solutions blowing-up
 at  $k$-dimensional lattice for $k\ge 1$.

\begin{lemma}\label{lem:4}
Suppose $N>2m+2, 1\leq k<\frac{N-2m}{2}$ and denote $\bar\mu=\min\{\mu_1,\cdots,\mu_n\}$. Then there exists $\tilde \theta>0$ small, such that
$$\begin{array}{ll}
\quad\displaystyle\int_{\mathbb{R}^{N}}\frac{1}{|y-z|^{N-2m}}W_{n}^{\frac{4m}{N-2m}}
(z)\sigma(z)\sum\limits_{j}\frac{\mu_j^{\frac{N-2m}2}}{(1+\mu_j|z-x_j|)
^{\frac{N-2m}{2}+\tau}}dz\\
\leq
\displaystyle C\sigma(y)\sum\limits_{j}\frac{\mu_j^{\frac{N-2m}2}}{(1+\mu_j|y-x_j|)^{\frac{N-2m}{2}+\tau+
\tilde \theta}}
+\frac{C}{\bar{\mu}^{\tilde\vartheta}}\sigma(y)\sum\limits_{j}\frac{\mu_j^{\frac{N-2m}2}}{(1+\mu_j|y-x_j|)^{\frac{N-2m}{2}+\tau}}.\end{array}$$
\end{lemma}
\begin{proof}

By Lemma \ref{lem:3}, if $z\in B_1(x_i)$,  we have
\begin{equation}\label{2.42}
\begin{aligned}
W_{n}^{\frac{4m}{N-2m}}\sum\limits_{j}\frac{\mu_j^{\frac{N-2m}2}}{(1+\mu_j|z-x_j|)
^{\frac{N-2m}{2}+\tau}}\leq C\frac{\mu_i^{\frac{N-2m}2+2m}}{(1+\mu_i|z-x_i|)^{\frac{N-2m}{2}+4m+\tau}}.
\end{aligned}
\end{equation}

By Lemma \ref{lem:2}, we have
\begin{equation}\label{2.45}
\begin{aligned}
&\quad
\int_{\Omega_{i}\cap B_i}\frac{1}{|y-z|^{N-2m}}W_{n}^{\frac{4m}{N-2m}}(z)\sigma(z)
\sum\limits_{j}\frac{\mu_j^{\frac{N-2m}2}}{(1+\mu_j|z-x_j|)
^{\frac{N-2m}{2}+\tau}}
dz\\
&\leq C\int_{\Omega_{i}\cap B_i}\frac{1}{|y-z|^{N-2m}}
\frac{\mu_i^{\frac{N-2m}2+2m}}{(1+\mu_i|z-x_i|)^{\frac{N-2m}{2}+4m+\tau}}
(\frac{1+\mu_i|z-x_i|}{\mu_i})^{\tau-1}\\
& \leq \frac{C\mu_i^{\frac{N-2m}2+1-\tau}}{(1+\mu_i |y-x_i|)^{\min(\frac{N-2m}{2}+2m+1, N-2m)}}\\
& \leq (\frac{1+\mu_i |y-x_i|}{\mu_i})^{\tau-1}\frac{C\mu_i^{\frac{N-2m}2}}{(1+\mu_i|y-x_i|)^{\frac{N-2m}{2}+\tau
+\tilde \theta}}.
\end{aligned}
\end{equation}
 Similarly, we have
\begin{equation}\label{2.46}
\begin{aligned}
&\quad
\int_{\Omega_{i}\cap B_i}\frac{1}{|y-z|^{N-2m}}W_{n}^{\frac{4m}{N-2m}}(z)\sigma(z)
\sum\limits_{j}\frac{\mu_j^{\frac{N-2m}2}}{(1+\mu_j|z-x_j|)
^{\frac{N-2m}{2}+\tau}}
dz\\
&\leq C\int_{\Omega_{i}\cap B_i}\frac{1}{|y-z|^{N-2m}}
\frac{\mu_i^{\frac{N-2m}2+2m}}{(1+\mu_i|z-x_i|)^{\frac{N-2m}{2}+4m+\tau}}
\\
&\leq \frac{C\mu_i^{\frac{N-2m}2}}{(1+\mu_i |y-x_i|)^{\frac{N-2m}{2}+\tau+\tilde \theta}},
\end{aligned}
\end{equation}
which, together with \eqref{2.45}, gives
\begin{equation}\label{2.47}
\begin{aligned}
&\quad
\int_{\Omega_{i}\cap B_i}\frac{1}{|y-z|^{N-2m}}W_{n}^{\frac{4m}{N-2m}}(z)\sigma(z)
\sum\limits_{j}\frac{\mu_j^{\frac{N-2m}2}}{(1+\mu_j|z-x_j|)
^{\frac{N-2m}{2}+\tau}}
dz\\
&\leq \frac{C\sigma(y)}{(1+|y-x_i|)^{\frac{N-2m}{2}+\tau+\tilde \theta}}.
\end{aligned}
\end{equation}

 We have the following inequality:

\begin{equation}\label{10-22-1}
\begin{split}
& \Bigl(
\sum_{j}  \frac{\mu_{j}^{\frac{N-2m}2}}{(1+ \mu_{j}|y-x_{j}|)^{ \frac{N-2m}2+\tau
}}\Bigr)^{m^*-1}\\
\le & C\sum_{j} \frac{\mu_{j}^{\frac{N+2m}2}}{(1+ \mu_{j}|y-x_{j}|)^{ \frac{N+2m}2+\tau
}}\Bigl(\sum_{j} \frac{1}{(1+ \mu_{j}|y-x_{j}|)^{ \tau
}}
\Bigr)^{\frac{4m}{N-2m}}\\
\le & C\sum_{j} \frac{\mu_{j}^{\frac{N+2m}2}}{(1+ \mu_{j}|y-x_{j}|)^{ \frac{N+2m}2+\tau
}}.
\end{split}
\end{equation}

  If  $z\in \mathbb R^N\setminus \cup_j B_1(x_j)$, then

  \[
  W_n (z) =O\bigl(\frac1{\bar\mu^\vartheta}\bigr)\sum\limits_{j}\frac{\mu_j^{\frac{N-2m}2}}{(1+\mu_j|z-x_j|)
^{\frac{N-2m}{2}+\tau}}.
  \]
Thus, we have

\begin{equation}\label{2.42}
\begin{split}
&W_{n}^{\frac{4m}{N-2m}}\sum\limits_{j}\frac{\mu_j^{\frac{N-2m}2}}{(1+\mu_j|z-x_j|)
^{\frac{N-2m}{2}+\tau}}
\le   \frac C{\bar\mu^\vartheta} \Bigl( \sum\limits_{j}\frac{\mu_j^{\frac{N-2m}2}}{(1+\mu_j|z-x_j|)
^{\frac{N-2m}{2}+\tau}}  \Bigr)^{m^*-1}
\\
\leq &   \frac C{\bar\mu^\vartheta}\sum_{j} \frac{\mu_j^{\frac{N+2m}2}}{(1+ \mu_{j}|z-x_{j}|)^{ \frac{N+2m}2+\tau
}}, \quad \forall\; z\in \mathbb R^N\setminus \cup_j B_1(x_j).
\end{split}
\end{equation}
Then, the result follows from Lemma~\ref{lem:2}.
\end{proof}

Set
\begin{equation}\label{10-22-80}
W_{L,{\bf x,\mu}}=\sum_{i=1}^{\infty }U_{x_{i, L},\mu_{i,L}}.
\end{equation}
 Define
\begin{equation}\label{10-22-8}
N(\omega_L)=K(y)\Bigl(\bigl(  W_{L,{\bf x,\mu}} +\omega_L \bigr)_+^{m^*-1}-
 W_{L,{\bf x,\mu}}^{m^*-1}-(m^*-1) W_{L,{\bf x,\mu}}^{m^*-2}\omega_L\Bigr),
\end{equation}
and
\begin{equation}\label{11-22-8}
l_L=K(y)
 W_{L,{\bf x,\mu}}^{m^*}-\sum_{j=1}^\infty U_{x_{j, L}, \mu_{j, L}}^{m^*-1}.
\end{equation}

  We now  estimate
 $N(\omega_L)$ and $l_L$.

\begin{lemma}\label{l-1-5-3}
If $N> 2m+2$, then
\[
\|N(\omega_L)\|_{**}\le C  \|\omega_L\|_*^{\min(m^*-1, 2)}.
\]
\end{lemma}

\begin{proof}

We have
\[
|N(\omega_L)|\le
\begin{cases}
C|\omega_L|^{m^*-1}, & N\ge 6m;\vspace{2mm}\\
CW_{L, {\bf x,\mu}}^{m^*-3}\omega_L^2+C|\omega_L|^{m^*-1},  &N<6m.
\end{cases}
\]
Since $\tau= \frac{N-2m}2-\vartheta$, $k<\frac{N-2m}2$ and $\vartheta>0$ is small, we find that if $N<6m$, then
\[
W_{L, {\bf x,\mu}}^{m^*-3}\le C\Bigl( \sum\limits_{j}\frac{\mu_{j,L}^{\frac{N-2m}2}}{(1+\mu_{j,L}|z-x_{j,L}|)
^{\frac{N-2m}{2}+\tau}}\Bigr)^{m^*-3}.
\]
As a result,

\[
|N(\omega_L)|\le C\sigma(z) \Bigl( \sum\limits_{j}\frac{\mu_{j,L}^{\frac{N-2m}2}}{(1+\mu_{j,L}|z-x_{j,L}|)
^{\frac{N-2m}{2}+\tau}}\Bigr)^{m^*-1} \|\omega_L\|_*^{\min(m^*-1, 2)}.
\]
The result follows \eqref{10-22-1}.

\end{proof}

\begin{lemma}\label{l-2-5-3}
If $N>
2m+2$, then
\[
\|l_L\|_{**}\le C \Bigl(\frac{1}{\mu_L^{\min(
\frac{N+2m}2-\tau, \beta-\tau+1)}}+\max_i |x_{i, L}-P_i|^\beta\Bigr),
\]
where $\mu_L=\min_j \mu_{j, L}$.
\end{lemma}

\begin{proof}

We have
\[
l_L= K(y)\Bigl(
W_{L, {\bf x, \mu}}^{m^*-1}-\sum_{j} U_{x_{j,L}, \mu_{j,L}}^{m^*-1}\Bigr)
+\bigl(
K(y)-1\bigr)\sum_{j} U_{x_{j,L}, \mu_{j,L}}^{m^*-1}
:= J_1+J_2.
\]

 Assume $y\in\Omega_i$. Then,
\begin{equation}\label{1-l2-5-3}
\begin{split}
|J_1|\le & C\frac{\mu_{i,L}^{2m}}{(1+\mu_{i,L}|y-x_{i,L}|)^{4m}}\sum_{j\ne i} \frac{\mu_{j,L}^{\frac{N-2m}2}}{(1+\mu_{j,L}|y-x_{j,L}|)^{N-2m}}\\
&+
C\Bigl(
\sum_{j\ne i} \frac{\mu_{j,L}^{\frac{N-2m}2}}{(1+\mu_{j,L}|y-x_{j,L}|)^{N-2m}}\Bigr)^{m^*-1}:= J_{11}+J_{12}.
\end{split}
\end{equation}
Since $
|y-x_{j, L}|\ge |y-x_{i, L}|
$  for $y\in\Omega_i$,
we find
\begin{equation}\label{2-l2-5-3}
|J_{11}|
\le \frac{C\mu_{i, L}^{\frac{N+2m}2}}{(1+\mu_{i,L}|y-x_{i,L}|
)^{\frac{N+2m}2+\tau} }\sum_{j\ne i}\frac1{(\mu_{i, L}|x_{j,L}-x_{i,L}|)^{
\frac{N+2m}2-\tau}},
\end{equation}
and
\begin{equation}\label{nn2-l2-5-3}
|J_{11}|
\le \frac{C\mu_{i, L}^{\frac{N+2m}2}}{(1+\mu_{i,L}|y-x_{i,L}|
)^{\frac{N+2m}2+\tau} }\bigl( \frac{  1+\mu_{i,L}|y-x_{i,L}| }{ \mu_i }  \bigr)^{\tau-1}\sum_{j\ne i}\frac
{\mu_{j,L}^{\tau-1}}{(\mu_{i, L}|x_{j,L}-x_{i,L}|)^{
\frac{N+2m}2-1}}.
\end{equation}

Hence, we obtain
\begin{equation}\label{100-1-2-9}
\|J_{11}\|_{**}\le  \frac{C}{(\mu_{ L}L)^{
\frac{N+2m}2-\tau}}.
\end{equation}
Similarly, we can also prove
\begin{equation}\label{101-1-2-9}
\|J_{12}\|_{**}\le  \frac{C}{(\mu_{ L}L)^{
\frac{N+2m}2-\tau}}.
\end{equation}
Thus, \eqref{100-1-2-9} and \eqref{101-1-2-9} yield
\begin{equation}\label{1-2-9}
\|J_1\|_{**}\le \frac{C}{(\mu_{ L}L)^{
\frac{N+2m}{2}-\tau}}.
\end{equation}

 Now, we estimate  $J_2$.  Similar to the proof of \eqref{1-2-9}, we have
\begin{equation}\label{6-2-9}
\sum_{j\ne i}U_{x_{j,L}, \mu_{j,L}}^{m^*-1}\le \frac{C\sigma(y) \mu_{i,L}^{\frac{N+2m}2}}{(1+\mu_{i,L}|y-x_{i,L}|)^{\frac{N+2m}2 + \tau }}\frac1{
(\mu_L  L)^{ \frac{N+2m}2 - \tau }},\quad y\in \Omega_i.
\end{equation}

  We also have
\begin{equation}\label{8-2-9}
|K(y)-1| U_{x_{i,L}, \mu_{i,L}}^{m^*-1}\le  \frac{C\mu_{i,L}^{\frac{N+2m}2}}{(1+\mu_{i,L}|y-x_{i,L}|)^{\frac{N+2m}2 + \tau }}\frac{1}{
\mu_{i,L}^{ \frac{N+2m}2 - \tau }},\quad |y-x_{i, L}|\ge 1.
\end{equation}
If $ |y-x_{i, L}|\le 1$,
\begin{equation}\label{7-2-9}
\begin{split}
&|K(y)-1| U_{x_{i,L}, \mu_{i,L}}^{m^*-1}\le C |y-P_i|^\beta U_{x_{i,L}, \mu_{i,L}}^{m^*-1}\\
\le & \frac{C\sigma(y)\mu_{i,L}^{\frac{N+2m}2}}{(1+\mu_{i,L}|y-x_{i,L}|)^{\frac{N+2m}2 + \tau }}
(\sigma(y))^{-1}\Bigl( \frac{
|x_{i, L}-P_i|^\beta+|y-x_{i, L}|^{\beta}}{
(1+\mu_{i,L}|y-x_{i,L}|)^{\frac{N+2m}2 - \tau }} \Bigr).
\end{split}
\end{equation}

 Noting  that for any $\tau_1\in [1, \frac{N-2m}2)$, we have
\[
\frac{|y-x_{i, L}|^{\beta}}{
(1+\mu_{i,L}|y-x_{i,L}|)^{\frac{N+2m}2 - \tau_1 }}\le
\begin{cases} \frac{C}{\mu_{i,L}^\beta},& \text{if}\; \beta   -\frac{N+2m}2+\tau_1\le 0;
\\
\frac{C}{\mu_{i,L}^{\frac{N+2m}2-\tau_1}},& \text{if}\; \beta   -\frac{N+2m}2+\tau_1>0.
\end{cases}
\]
Hence,  we have proved
\begin{equation}\label{9-2-9}
\|J_2\|_{**}\le \frac{C}{\mu_{ L}^{\min(
\frac{N+2m}2-\tau, \beta-\tau+1)}}+C\max_i |x_{i, L}-P_i|^\beta.
\end{equation}

\end{proof}

\section{Asymptotic energy expansion}

\begin{lemma}\label{lemma:2.1}.  There are constants $C_1\ne 0$ and $C_2>0$, such that
\begin{eqnarray}
 \int_{\mathbb{R}^{N}}K(y)U_i^{m^*-1}\frac{\partial U_i}{\partial x_{i,j}}&=&
\frac{a_j  C_1}{\mu_i^{\beta-2}} (P_{i,j}- x_{i, j}  )+O\bigl( \frac1{\mu_i^{\beta-1+\theta}}+  |x_i-P_i   |^{\beta-1}  \bigr),\label{lemma:2.1e1}\\
\int_{\mathbb{R}^{N}}K(y)U_i^{m^*-1}\frac{\partial U_i}{\partial \mu_i}&=&
-\frac{C_2 \sum_{j=1}^N a_j}{\mu_i^{\beta+1}}+ O\bigl( \frac1{\mu_i^{\beta+1+\theta}}
+
\frac{ |x_i- P_i|^{\beta-2} }{\mu_i^{3}} \bigr).\label{lemma:2.1e2}
\end{eqnarray}
\end{lemma}

\begin{proof} We first prove \eqref{lemma:2.1e1}. We have
\begin{eqnarray*}
\int_{\mathbb{R}^{N}}K(y)U_i^{m^*-1}\frac{\partial U_i}{\partial x_{i,j}}&=&
\frac1{m^*}\int_{\mathbb{R}^{N}} \frac{\partial K(y)}{\partial y_j}U_i^{m^*}\\
&=&
\frac{\beta}{m^*}
\int_{\mathbb{R}^{N}}  a_j | \mu_i^{-1} y_j + P_{i,j}- x_{i, j}   |^{\beta-2}
 ( \mu_i^{-1} y_j + P_{i,j}- x_{i, j} )  U_{0,1}^{m^*}
 \\
 &&+O\bigl( \frac1{\mu_i^{\beta-1+\theta}}+  |x_i- P_i   |^{\beta-1+\theta}  \bigr)
\\
&=& \frac{a_j  C_1}{\mu_i^{\beta-2}} (P_{i,j}- x_{i, j}  )+O\bigl( \frac1{\mu_i^{\beta-1+\theta}}+  |x_i-P_i   |^{\beta-1}  \bigr).
\end{eqnarray*}

\vskip8pt

 Now we prove \eqref{lemma:2.1e2}.  Using
\[
\int_{\mathbb{R}^{N}}U_i^{m^*-1}\frac{\partial U_i}{\partial \mu_i}=0,
\]
we find
\begin{eqnarray*}
 \int_{\mathbb{R}^{N}}K(y)U_i^{m^*-1}\frac{\partial U_i}{\partial \mu_i}&=&
\int_{\mathbb{R}^{N}}\sum_{j=1}^N a_j |y_j - P_{i,j}|^\beta U_i^{m^*-1}\frac{\partial U_i}{\partial \mu_i}
+O\bigl( \frac1{\mu_i^{\beta+1+\theta}}  \bigr)
\\
&=&\displaystyle \frac{1}{m^*}
\int_{\mathbb{R}^{N}}\frac{\partial}{\partial\mu_i}\sum_{j=1}^N a_j |\mu_i^{-1} y_j+ x_{i,j} - P_{i,j}
|^\beta U_{0,1}^{m^*}
\,dy+O\bigl( \frac1{\mu_i^{\beta+1+\theta}}  \bigr)\\
&=&-\frac{C_2 \sum_{j=1}^N a_j}{\mu_i^{\beta+1}}+ O\bigl( \frac1{\mu_i^{\beta+1+\theta}} +
\frac{ |x_i- P_i|^{\beta-2} }{\mu_i^{3}} \bigr).
\end{eqnarray*}

 We complete the proof of Lemma \ref{lemma:2.1}.
\end{proof}

\begin{lemma}\label{lemma:2.4} Set $W(x)=\sum\limits_{i}U_i  $. Here the sum can be from
 $1$ to $n$, or from 1 to infinity.  Then for $j=1, \cdots, N$,
\begin{eqnarray*}
&&\int_{\mathbb{R}^{N}}(-\Delta )^m W  Z_{i,j}dy
-\int_{\mathbb{R}^{N}}K(y)W^{m^*-1}Z_{i,j}dy\\
&=&\frac{a_j  C_1}{\mu_i^{\beta-2}} (P_{i,j}- x_{i, j}  )+O\bigl( \frac1{\mu_i^{\beta-1+\theta}}+  |x_i-P_i   |^{\beta-1+\theta} +\frac{1}{\mu_L^{N-2m}L^{N-2m+1}}\bigr).
\end{eqnarray*}

\end{lemma}

\begin{proof}

We have
\begin{equation}\label{3.48}
\int_{\mathbb{R}^{N}}(-\Delta )^m W  Z_{i,j}dy=\int_{\mathbb{R}^{N}}\sum\limits_{l}U_l^{m^*-1} Z_{i,j}dy
=O\bigl(\frac{1}{\mu^{N}L^{N+2m+1}}\bigr).
\end{equation}

On the other hand,
\begin{equation}\label{3.47}
\begin{split}
&\quad\quad\int_{\mathbb{R}^{N}}K(y)W^{m^*-1}Z_{i,j}dy-\int_{\mathbb{R}^{N}}K(y)U_i^{m^*-1}Z_{i,j}dy
\\
&=(m^*-1) \int_{\mathbb{R}^{N}}K(y)U_i^{m^*-2}Z_{i,j}\sum_{l\ne i} U_l +O\Bigl(\int_{\mathbb{R}^{N}}|Z_{i,j}|
\bigl(\sum_{l\ne i} U_l\bigr)^{m^*-1}\Bigr)\\
&= O\bigl(\frac{1}{\mu^{N-2m}L^{N-2m+1}}\bigr).
\end{split}
\end{equation}

 Combining \eqref{3.48}, \eqref{3.47} and Lemma~\ref{lemma:2.1},
 we obtain the desired result in Lemma \ref{lemma:2.4}.
\end{proof}

\begin{lemma}\label{lemma:2.5} There exists some constant $C>0$ independent of $i, j, n,$ such that
\begin{eqnarray*}
&&\int_{\mathbb{R}^{N}} (-\Delta )^m W  Z_{i,N+1}dy
-\int_{\mathbb{R}^{N}}K(y)W^{m^*-1}Z_{i,N+1}dy\\
&=&\frac{C_2 \sum_{j=1}^N a_j}{\mu_i^{\beta+1}}+\sum_{l\ne i }\frac{C_4}{\mu_i(\mu_i\mu_l)^{\frac{N-2m}{2}} |x_{ i}-x_{ l}|^{N-2m}}\\
&&+ O\bigl( \frac1{\mu_i^{\beta+1+\theta}}  +  |x_i- P_i   |^{\beta-1+\theta}  +\frac{1}{\mu_L^{N-2m+1+\theta}L^{N-2m+\theta}}\bigr).
\end{eqnarray*}

\end{lemma}

\begin{proof}

Similar to \eqref{3.47}, we can deduce
\begin{eqnarray*}
&&\int_{\mathbb{R}^{N}}K(y)W^{m^*-1}Z_{i,N+1}dy-\int_{\mathbb{R}^{N}}K(y)U_i^{m^*-1}Z_{i,N+1}dy
\\
&=&(m^*-1)\int_{\mathbb{R}^{N}}K(y)U_i^{m^*-2}Z_{i,N+1}\sum_{l\ne i} U_l + O\bigl(\frac{1}{\mu^{N}L^{N+2m+1}}\bigr)\\
&=& (m^*-1) \int_{\mathbb{R}^{N}}U_i^{m^*-2}Z_{i,N+1}\sum_{l\ne i} U_l\\
&&+ O\bigl( \frac1{\mu_i^{\beta+1+\theta}}  +  |x_i- P_i    |^{\beta-1+\theta}  +\frac{1}{\mu_L^{N-2m+1+\theta}L^{N-2m+\theta}}\bigr),
\end{eqnarray*}
which, together with Lemma~\ref{lemma:2.1}, gives the result.

\end{proof}

\section{Estimate of the error term}

 Let $u_L$ be a solution of $(P)$ with the form
\begin{equation}\label{1-21-8}
u_L= W_{L, x, \mu}+\omega_L,\quad W_{L, x, \mu}= \sum_{j=1}^{+\infty} U_{x_{j,L}, \mu_{j,L}},
\end{equation}
satisfying \eqref{3-6-1}, \eqref{2-6-1} and \eqref{2-1-9}.
In this section, we will estimate the error term $\omega_L$.

  It is easy to see that $\omega_L$ satisfies the following equation:
\begin{equation}\label{re1}
(-\Delta)^m \omega_L
-(m^*-1)K(y)
 W_{L,x,\mu}^{m^*-2}\omega_L =  N(\omega_L)+l_L,
\end{equation}
where  $l_L$ and $N(\omega_L)$  are defined in \eqref{10-22-8}  and \eqref{11-22-8} respectively.

 By  assumption~\eqref{3-6-1}, we have
 $\|\omega_L\|_{*} \to 0$  as $L\to +\infty$.

\vskip8pt

 Note that in the decomposition~\eqref{1-21-8}, we do not  assume that $\omega_L\in \mathbf{H}_n$.  See
 \eqref{100-15-1} for the definition of $\mathbf{H}_n$. Let us point out that
$x_{j, L}$ and $\tilde C_m \mu_{j, L}^{\frac{N-2m}2}$ may not be a maximum point and
the maximum value of $u_L$ in $B_\delta(x_{j, L}),$ respectively.
Let $\bar x_{j, L}\in B_1(P_j)$ be such that $u_L(\bar x_{j, L})= \max_{B_1(P_j)}u_L:=
\tilde C_m\bar \mu_{j, L}^{\frac{N-2m}2}$.
 From \eqref{3-6-1}, we can deduce that as $L\to +\infty$,
$\bar\mu_{j, L}= \mu_{j, L}( 1+ o_L(1)),$
and $\mu_{j, L} (\bar x_{j, L}- x_{j, L}) =o_L(1).$  As a result, we have
\[
   \begin{split}
  & U_{ x_{j, L},  \mu_{j, L}}-U_{\bar x_{j, L}, \bar \mu_{j, L}}\\
  =& O\bigl( |\mu_{j, L} (\bar x_{j, L}- x_{j, L})|+ \mu_L^{-1} |\bar{\mu}_{j, L}- \mu_{j, L}| \bigr)U_{\bar x_{j, L}, \bar \mu_{j, L}}
   =o_L(1)U_{\bar x_{j, L}, \bar \mu_{j, L}},
   \end{split}
   \]
and
\[
|\omega_L(x)|= o_L(1)\sum_{j=-\infty}^\infty \frac{\bar \mu_{j,L}^{\frac{N-2m}2}}{(1+ \bar \mu_{j,L}|y-\bar x_{j,L}|)^{ \frac{N-2m}2+\tau
}}.
\]
So, we find that in \eqref{1-21-8}, $x_{j, L}$ and $\mu_{j, L}$ can be replaced by $\bar x_{j, L}$ and
$\bar \mu_{j,L}$ respectively.
 For simplicity, in the following, we still use $x_{j, L}$ and $\mu_{j, L}$ to denote $\bar x_{j, L}$ and  $\bar \mu_{j, L}$ respectively.  Thus in \eqref{1-21-8}, it holds
\begin{equation}\label{1-22-8}
| \omega_{L}(x_{j, L})|=O\bigl(\frac1{\mu_{j, L}^{\frac{N-2m}2} L^{N-2m}}\bigr),\quad
| \nabla \omega_{L}(x_{j, L})|=O\bigl(\frac1{\mu_{j, L}^{\frac{N-2m}2} L^{N-2m+1}}\bigr).
\end{equation}
In the following, we will use \eqref{re1} to estimate $\omega_L$.

\begin{proposition}\label{p1-6-3} It holds
\[
\|\omega_L\|_{*}\le \frac{C}{\mu_{ L}^{
\min(
\frac{N+2m}2-\tau, \beta-\tau+1)}}+C\max_i |x_{i, L}-P_i|^\beta.
\]
\end{proposition}

\begin{proof}
 In the existence part, we already know that
\begin{equation}\label{3-3-9}
\begin{split}
&|\omega_L(y)|\Bigl(\sigma(y)\sum\limits_{i}
\frac{\mu_{i,L}^{\frac{N-2m}2}}{(1+\mu_{i,L}|y-x_{i,L}|)^{\frac{N-2m}{2}+\tau}}\Bigr)^{-1}\\
\le & C \bigl( \|\omega_L\|_*^{\min(m^*-1, 2)}+\|l_L\|_{**}\bigr)
+C\|\omega_L\|_*
\frac{\sum\limits_{i}\frac{1}{(1+\mu_{i,L}|y-x_{i,L}|)^{\frac{N-2m}{2}+\tau+\tilde \theta}}}{\sum\limits_{i}
\frac{1}{(1+\mu_{i,L}|y-x_{i,L}|)^{\frac{N-2m}{2}+\tau}}} ,
\end{split}
\end{equation}
where $\tilde \theta>0$ is a constant.

 Suppose that there is $y\in \mathbb R^N \setminus \cup_{j=1}^\infty B_{R\mu_{j, L}^{-1}}(x_{j, L})$
for some large $R>0$,
such that $\|\omega\|_*$ is achieved at $y$. Then, \eqref{3-3-9} gives
\begin{equation}\label{4-3-9}
\|\omega_L\|_*\le  C \bigl( \|\omega_L\|_*^{\min(m^*-1, 2)}+\|l_L\|_{**}\bigr)
+o_R(1)\|\omega_L\|_*.
\end{equation}
Since $\|\omega_L\|_*\to 0$ as $L\to +\infty$, we obtain
\begin{equation}\label{5-3-9}
\|\omega_L\|_*\le  C \|l_L\|_{**}.
\end{equation}

 Suppose that $\|\omega\|_*$ is achieved at $y\in B_{R\mu_{j, L}^{-1}}(x_{j, L})$.
Let
\[
\tilde \omega_L(y)=\mu_{j, L}^{-\frac{N-2m}2}\omega_L(\mu_{j, L}^{-1} y+ x_{j, L}),
\]
and
\[
|||\tilde \omega_L|||_{*} = \sup_{y \in \R^N }\Bigl(\sigma(\mu^{-1}_{j,L}y+x_{j,L})
\sum_{i=1}^\infty \frac{\mu_{j, L}^{-\frac{N-2m}2}\mu_{i, L}^{\frac{N-2m}2}}{(1+ \frac{\mu_{i,L}}{\mu_{j,L}}
|y-\mu_{j, L}(x_{i, L}-x_{j,L})|
)^{ \frac{N-2m}2+\tau
}}\Bigr)^{-1}|\tilde \omega_L(y)|.
\]
Then $|||\tilde \omega_L|||_{*}$ is achieved at some $y\in B_R(0)$.

 Suppose that
\[
|| \omega_L||_{*}\ge N_L\Bigl(\frac{1}{\mu_L^{
\min(
\frac{N+2m}2-\tau, \beta-\tau+1)}}+C\max_i |x_{i, L}-P_i|^\beta\Bigr),
\]
for some $N_L\to \infty$.  Then
as $L\to +\infty$, $\eta_L =\frac{\tilde \omega_L}{|||\tilde \omega_L|||_{*}}$ converges to $\eta\ne 0$, which satisfies
\[
(-\Delta)^m \eta -(m^*-1) U_{0,1}^{m^*-2}\eta=0,\quad \text{in}\; \mathbb R^N,
\]
since    $\frac{\|l_L\|_{**}}{ ||\tilde \omega_L||_{*}   }\le \frac{C}{ N_L}    \to 0$. This gives
\[
\eta=\alpha_0 \frac{\partial U_{0,\mu}}{\partial \mu}\bigr|_{\mu=1}+\sum_{j=1}^N \alpha_j\frac{\partial U_{0, 1}
}{\partial x_j},
\]
for some constant $\alpha_j$.

 On the other hand, we have
\begin{eqnarray*}
\tilde \omega_L(0)&=&\mu_{j, L}^{-\frac{N-2m}2}\omega_L( x_{j, L})=O\bigl(\frac1{\mu_L^{N-2m}L^{N-2m}}\bigr),\\
\nabla \tilde \omega_L(0)&=&\mu_{j, L}^{-\frac{N-2m+2}2}\nabla\omega_L( x_{j, L})=O\bigl(\frac1{\mu_L^{N-2m+1}L^{N-2m+1}}\bigr).
\end{eqnarray*}
So, we find  $\eta(0)=0$ and $\nabla \eta(0)=0$, which implies $\alpha_0=\alpha_1=\cdots=\alpha_N=0$. This is a contradiction.

\end{proof}

\begin{corollary}\label{c1-5-9}

For any $\delta>0$, we have
\[
\|\omega_L(x)\|_{C^{2m-1}(B_{2\delta}(x_{j, L})\setminus B_{\frac12\delta}(x_{j, L}))}\le\frac{C}{\mu_L^\tau}\Bigl( \frac{1}{\mu_{ L}^{
\min(
\frac{N+2m}2-\tau, \beta-\tau+1)}}+\max_i |x_{i, L}-P_i|^\beta\Bigr).
\]

\end{corollary}

\begin{proof}

It follows from Proposition~\ref{p1-6-3} that
\begin{eqnarray*}
|\omega_L(x)|& \le & \|\omega_L(x)\|_* \sum_{j=1}^\infty \frac{\mu_{j,L}^{\frac{N-2m}2}}{(1+ \mu_{j,L}|y-x_{j,L}|)^{ \frac{N-2m}2+\tau
}}\\
&\le &\frac{C}{\mu_L^\tau}\Bigl( \frac{1}{\mu_{ L}^{
\min(
\frac{N+2m}2-\tau, \beta-\tau+1)}}+\max_i |x_{i, L}-P_i|^\beta\Bigr), \quad x\in B_{4\delta}(x_{j, L})\setminus B_{\frac14\delta}(x_{j, L}).
\end{eqnarray*}
On the other hand, from \eqref{re1}, using the $L^p$ estimates, we can deduce that for any $p>1$,
\begin{eqnarray*}
&&\|\omega_L\|_{W^{2m, p}(B_{2\delta}(x_{j, L})\setminus B_{\frac12\delta}(x_{j, L}))}\\
&\le & C\|\omega_L\|_{L^\infty(B_{4\delta}(x_{j, L})\setminus B_{\frac14\delta}(x_{j, L}))}\\
&&+
C\|(m^*-1)K(y)
 W_{L,x,\mu}^{m^*-2}\omega_L +  N(\omega_L)+l_L\|_{L^\infty(B_{4\delta}(x_{j, L})\setminus B_{\frac14\delta}(x_{j, L}))}
\\
&\le & C\|\omega_L\|_{L^\infty(B_{4\delta}(x_{j, L})\setminus B_{\frac14\delta}(x_{j, L}))}+C \Bigl(
\|  N(\omega_L)\|_{**}+\|l_L\|_{**}\Bigr)\frac1{\mu^\tau}.
\end{eqnarray*}
The result follows from Lemmas~\ref{l-1-5-3}
and \ref{l-2-5-3}.
\end{proof}

\section{Some basic lemmas}

  For any integer $n\ge 2$, consider the following equations:
\begin{equation}\label{1-13-9}
|a_j|^{\beta-1} a_j- \sum_{i=1}^n  d_{ji}a_{i}=0, \quad j=1,\cdots, n,
\end{equation}
where $\beta>1$ is a constant,  $d_{ij}$ satisfies  $d_{ii}=0$, $i=1, \cdots, k$, $d_{ij}>0$, $d_{ij}=d_{ji}$,
$i\ne j$,  and $c_1\ge \max_j \sum_{i=1}^n  d_{ji}\ge\min_j \sum_{i=1}^n  d_{ji}\ge c_0>0$.
It is easy to see that if $a=(a_1, \cdots, a_k)$ satisfies \eqref{1-13-9}, then $a$ is a critical
point of the function defined as
\begin{equation}\label{0-13-9}
F(x)=\frac1{\beta+1} \sum_{j=1}^n |x_j|^{\beta+1}-\frac12\sum_{i=1}^n\sum_{j=1}^n d_{ij}x_ix_j.
\end{equation}

It is easy to show that $\min_{x\in \mathbb R^n} F(x)<0$ is achieved at $a\ne 0$. Moreover, from $d_{ij}\ge 0$,
we can assume $a_j\ge 0$, $j=1, \cdots, n$. Using \eqref{1-13-9}, if $a_j=0$ for some $j$, then
$a_i=0$ for all $i\ne j$. Moreover, from \eqref{1-13-9}, we find
\begin{equation}\label{100-1-13-9}
(\max_j a_j)^{\beta-1} \le \max_j \sum_{i=1}^n  d_{ji},\quad
(\min_j a_j)^{\beta-1} \ge \min_j \sum_{i=1}^n  d_{ji}\ge c_0>0.
\end{equation}
We now prove the following result.

\begin{lemma}\label{l1-13-9}

The solution of \eqref{1-13-9} is unique if $a_j>0$, $j=1, \cdots, n$.
Moreover, if we define the linear operator $A$ as follows:
\begin{equation}\label{2-13-9}
(A X)_j= \beta a_j^{\beta-1} x_j - \sum_{i=1}^n d_{ji} x_i, \quad j=1, \cdots, n.
\end{equation}
Then  $\| A X\|\ge c' \|X\|$ for some $c'>0$, where the norm for $X$ is defined as $\|X\|=\max_j |x_j|$.
\end{lemma}

\begin{proof}

This lemma was proved in \cite{DLY}. For the readers' convenience, we give its proof here.

 Suppose that \eqref{1-13-9} has two solutions $a=(a_1, \cdots, a_n)$ and $b=(b_1, \cdots, b_n)$,
$a_j>0$, $b_j>0$, $j=1, \cdots, n$.  Let
$T=\max_{j}\frac{a_j-b_j}{a_j}$. We have two possibilities. (i) $T\le 0$,  (ii) $T>0$.

 If $T\le 0$, then $a_j\le b_j$ for all $j=1, \cdots, n$. So we can define $T_1=\sup_j\frac{t_j}{a_j}$, $t_j= b_j-a_j\ge 0$. Then $b_j^\beta \ge a_j^\beta +\beta a_j^{\beta-1}(b_j-a_j)$.
So
\[
\begin{split}
\beta a_j^{\beta-1}(b_j-a_j)\le \sum_{i=1}^n d_{ji} (b_i- a_i)\le T_1\sum_{i=1}^n d_{ji} a_i= T_1 a_j^\beta,
\end{split}
\]
which implies $\beta T_1\le T_1$. So $T_1=0$.

 If $T>0$, then there is a $j$, such that $\frac{a_j-b_j}{a_j}=T>0$. As a result,
\[
\beta a_j^{\beta-1}(a_j-b_j)\le \sum_{i=1}^n d_{ji} (a_i- b_i)\le T\sum_{i=1}^n d_{ji} a_i= T a_j^\beta,
\]
which implies $\beta T\le T$. So $T=0$. This is a contradiction.

 To prove the last part,  for any $X$ with $\|X\|=1$, we let
$
 x=\sup_j\frac{|x_j|}{a_j}$.
   Then
\[
|\sum_{i=1}^n d_{ji} x_i|\le x  \sum_{i=1}^n d_{ji} a_i=x a_j^{\beta}.
 \]
As a result,
\[
| (A X)_j|\ge \beta a_j^{\beta-1} |x_j| -x a_j^{\beta}= a_j^{\beta} \bigl( \beta \frac{|x_j|}{a_j}- x\bigr).
\]
 Since $\beta>1$, we can choose $j$, such that
\[
| (A X)_j|\ge a_j^{\beta} \bigl( \beta \frac{|x_j|}{a_j}- x\bigr)\ge c'>0.
\]

\end{proof}

Now we consider
\begin{equation}\label{100-1-13-9}
a_j^{\beta} - \sum_{i=1}^n  d_{ji}a_{i}=0, \quad j=1,\cdots, .
\end{equation}
Using Lemma~\ref{l1-13-9}  and \eqref{100-1-13-9}, we can easily prove the following result.

\begin{lemma}\label{100-l1-13-9}

Equation \eqref{100-1-13-9} has a  unique solution $a_j>0$, $j=1, \cdots, $.
Moreover, if we define the linear operator $A$ as follows:
\begin{equation}\label{100-2-13-9}
(A X)_j= \beta a_j^{\beta-1} x_j - \sum_{i=1}^\infty d_{ji} x_i, \quad j=1, \cdots, .
\end{equation}
Then  $\| A X\|\ge c' \|X\|$ for some $c'>0$, where the norm for $X$ is defined as $\|X\|=\max_j |x_j|$.
\end{lemma}

\medskip

\textbf{Acknowledgment.}  The first and second authors are supported by NSFC(11171171, 11331010,11328101). The third author is partially supported
by ARC (DP130102773).

\end{document}